%% file: main.tex
\documentclass[11pt]{article}

\usepackage[utf8]{inputenc} 
\usepackage[T1]{fontenc}
\usepackage{lmodern}
\usepackage{hyperref}       
\usepackage{url}            
\usepackage{booktabs}       
\usepackage{amsfonts}       
\usepackage[numbers]{natbib}
\usepackage{nicefrac}       
\usepackage{microtype}      
\usepackage{algorithm,algpseudocode}
\usepackage{amsfonts}
\usepackage{amsmath}
\usepackage{amsthm}
\usepackage{amssymb}
\usepackage[title]{appendix}
\usepackage{bm}
\usepackage{courier}
\usepackage[usenames,dvipsnames]{color}
\usepackage{enumitem}
\usepackage{graphicx}
\usepackage[margin=1in,letterpaper]{geometry}
\usepackage{url}
\usepackage{subfigure}
\usepackage{multirow}
\usepackage[dvipsnames]{xcolor}
\usepackage{setspace}
\usepackage{caption}
\hypersetup{
	colorlinks,
	linkcolor={red!80!black},
	citecolor={blue!50!black},
	urlcolor={blue!80!black}
}

\input{./Definitions}

\graphicspath{ {./figures/} }
\DeclareGraphicsExtensions{.pdf, .jpeg, .png}

\usepackage{parskip}

\title{
On the Convergence of Inexact Predictor-Corrector Methods for Linear Programming
}
\date{}
\author{
  Gregory Dexter\footnote{\scriptsize Department of Computer Science, Purdue University,
West Lafayette, IN, USA, \url{{gdexter, pdrineas}@purdue.edu}.}
  \and
  Agniva Chowdhury\footnote{\scriptsize Computer Science and Mathematics Division, Oak Ridge National Laboratory, TN, USA,\,\url{chowdhurya@ornl.gov}.}~\thanks{\scriptsize This work was done when the author was a graduate student at Purdue University.}
  \and
  Haim Avron\footnote{\scriptsize School of Mathematical Sciences, Tel Aviv University, Tel Aviv, Israel, \url{haimav@tauex.tau.ac.il}.}
  \and
  Petros Drineas\footnotemark[1]
}

\begin{document}

\maketitle

\begin{abstract}
Interior point methods (IPMs) are a common approach for solving linear programs (LPs) with strong theoretical guarantees and solid empirical performance. The time complexity of these methods is dominated by the cost of solving a linear system of equations at each iteration. In common applications of linear programming, particularly in machine learning and scientific computing, the size of this linear system can become prohibitively large, requiring the use of iterative solvers, which provide an approximate solution to the linear system. However, approximately solving the linear system at each iteration of an IPM invalidates the theoretical guarantees of common IPM analyses. To remedy this, we theoretically and empirically analyze (slightly modified) predictor-corrector IPMs when using approximate linear solvers: our approach guarantees that, when certain conditions are satisfied, the number of IPM iterations does not increase \textit{and} that the final solution remains feasible. We also provide practical instantiations of approximate linear solvers that satisfy these conditions for special classes of constraint matrices using randomized linear algebra.
\end{abstract}

\input{1_introduction}

\input{2_notation_background}

\input{3_overview}

\input{5_pcc}

\input{main_experiments}

\input{conclusions}

\section*{Acknowledgements}

GD, AC, and PD were partially supported by NSF 10001390, NSF 10001415, and DOE 14000600. HA was partially supported by BSF grant 2017698. 

\setlength{\bibsep}{6pt}
\bibliographystyle{plainnat}
{
\bibliography{bibliography}
}

\appendix

\input{appendix_a}

\input{4_pcu}

\input{adj_proofs}

\input{solver}

\input{experiments}

\end{document}

%% file: 1_introduction.tex
\section{Introduction} \label{section:intro}

Linear programming is a ubiquitous problem appearing across applied mathematics and computer science, with extensive applications in both theory and practice. Modern machine learning applications of linear programming include $\ell_1$-regularized SVMs~\citep{zhu20041}, basis pursuit (BP)~\citep{yang2011alternating}, sparse inverse covariance matrix estimation (SICE)~\citep{yuan2010high}, the nonnegative matrix factorization (NMF)~\citep{recht2012factoring}, MAP inference~\citep{meshi2011alternating}, compressed sensing \citep{donoho2006compressed},
and adversarial deep learning \citep{wong2018provable}. The central importance of this problem has resulted in substantial research on provably accurate algorithms for linear programming, at the same time, practically efficient algorithms are critically needed. 

The two major families of algorithms used to solve linear programs are simplex methods and interior point methods (IPMs), with combinations of the two (e.g., IPMs used in the early stages and simplex methods used once approximately optimal solutions have been reached) being useful in practice \cite{wright1997primal}. Predictor-corrector methods, a special type of IPMs, have been particularly useful in solving linear programs accurately and are perhaps the most successful example of theoretically provable yet practically efficient approaches for linear programs. 

More precisely, consider a linear program (LP) of the following (standard) form. Let $\Ab \in \R^{m \times n}$ be the constraint matrix 
and $\xb \in \R^n$ be the free variable:
\begin{gather}\label{eq:primal_lp_def}
    \min \cbb^T \xb, \text{ subject to } \Ab\xb = \bb, ~\xb \geq \zero.
\end{gather}
The associated dual problem is
\begin{gather}\label{eq:dual_lp_def}
    \max \bb^T\yb, \text{ subject to } \Ab^T\yb + \sbb = \cbb, ~\sbb \geq \zero,
\end{gather}
%
where $\yb \in \R^m$ is the dual variable and $\sbb \in \R^n$ is the slack variable. The first (weakly) polynomial time algorithm for linear programming is the ellipsoid method, developed by Khachiyan in 1979~\cite{khachiyan1979polynomial}. While the ellipsoid method was deemed to be inefficient in practice, it provided inspiration for the first IPM, developed by Karmarkar in 1984~\cite{karmarkar84}. Karmarkar's initial work was followed by an explosion of research on IPMs that led to numerous algorithms with various theory-practice tradeoffs. 

Predictor-corrector IPMs achieve nearly optimal theoretical guarantees, while being commonly used in popular linear optimization packages~\cite{schork2020implementation,almeida2015convergence}. More precisely, predictor-corrector algorithms are primal-dual path-following IPMs. They compute a sequence of iterates $(\xb^k, \yb^k, \sbb^k)$ within the primal-dual polytope of feasible solutions which approach an optimal solution of the LP.  Path-following IPMs require that the iterates within the polytope remain near the so-called central path of the polytope, which results in faster convergence. In each iterate, updates are computed by solving the \emph{normal equations}, namely a system of linear equations of the following form:
\begin{gather}\label{eq:intro_normal_eq}
    \Ab\Db^2\Ab^T\dyexact = \pb.
\end{gather}
In the above equation, $\pb$ is a vector (see eqn.~(\ref{eq:p_vec_def}) for the exact definition) and $\Db^2 = \Xb\Sb^{-1}$, where $\Xb$ is the diagonal matrix whose entries are the $\xb_i$ and $\Sb$ is the diagonal matrix whose entries are the $\sbb_i$. To analyze the computational complexity of predictor-corrector methods, one first computes the number of \emph{outer iterations}, namely the number of iterations in the IPM algorithm required to converge to an approximately optimal solution. Then, one analyzes the time required to compute each of the iterates by solving the linear system of eqn.~(\ref{eq:intro_normal_eq}). 

Standard approaches analyzing the rate of convergence and the time complexity of predictor-corrector methods (and other IPMs) typically assume that eqn.~(\ref{eq:intro_normal_eq}) is solved exactly at each iteration.  However, this assumption becomes untenable for large-scale problems and inexact iterative solvers are nearly universally used in practice. The resulting methods are often called \textit{inexact} predictor-corrector IPMs. Theoretically understanding the behavior of inexact linear equation solvers when combined with IPMs is highly non-trivial. Indeed, an early, provably accurate, approach combining inexact solvers with short step IPMs (a different, less practical, class of IPMs) appeared in the work of~\citet{daitch2008faster}, which argued that eqn.~(\ref{eq:intro_normal_eq}) can be solved in near linear time when the constraint matrix $\Ab\Db^2\Ab^T$ is symmetric and diagonally dominant. This allowed fast approximate solutions to problems such as generalized maximum flow \cite{daitch2008faster}. More recently, the literature survey by~\citet{gondzio2012interior} highlighted that pairing IPMs with iterative linear solvers is the way forward towards solving large-scale LPs that arise in machine learning applications. See Section~\ref{section:related_work} for a detailed discussion of relevant prior work on LP solvers.

\subsection{Our contributions}

In this paper, we prove that a (slightly modified) predictor-corrector IPM can tolerate errors in solving the linear system of eqn.~(\ref{eq:intro_normal_eq}) at each outer iteration \textit{without} sacrificing the feasibility of the derived solution and \textit{without} increasing the number of outer iterations of predictor-corrector IPMs.
Our proposed inexact predictor-corrector IPM (Algorithm~\ref{algo:pcc} in Section~\ref{section:pcc}) starts with a feasible point and converges to an $\epsilon$-optimal exactly \textit{feasible} solution in $\Ocal(\sqrt{n} \log \frac{\mu_0}{\epsilon})$ outer iterations, where $\mu_0$ is the duality measure at the starting point\footnote{The \emph{duality measure} $\mu = \frac{1}{n}\xb^T\sbb$ quantifies how close a primal-dual point $(\xb, \yb, \sbb)$ is to optimality at a certain iteration.}, while approximately solving the system of linear equations of eqn.~(\ref{eq:intro_normal_eq}). 

Approximately solving the linear system of eqn.~(\ref{eq:intro_normal_eq}) is problematic for two reasons: first, it invalidates known analyses of classical predictor-corrector IPMs and, second, it results in infeasible iterates, even when the IPM starts from a feasible point. To address these issues in theory and in practice, we introduce an error-adjustment vector, similar to the work of~\citet{monteiro2003convergence,chowdhury2020speeding}. The error adjustment vector is our only modification to the classical predictor-corrector IPM and it allows us to return provably accurate, \textit{feasible} solutions, without any increase in the outer iteration complexity of predictor-corrector IPMs.

More precisely, let $\dytilde$ be an approximate solution for the linear system of eqn.~(\ref{eq:intro_normal_eq}) and let $\vb  \in \R^n$ be an \textit{error adjustment vector}
(more on this vector $\vb$ later). Let these two vectors satisfy the following conditions:
\begin{gather}\label{eq:solvev_guarantee}
     \Ab\Db^2\Ab^T \dytilde = \pb + \Ab\Sb^{-1}\vb
    ~~~\text{and}~~~ \|\vb\|_2 < \Theta(\epsilon).
\end{gather}
In words, the above conditions simply state that the approximate solution $\dytilde$ is an exact solution to a slightly modified system of normal equations, where the vector $\pb$ has been replaced by the vector $\pb + \Ab\Sb^{-1}\vb$. The two norm of the vector $\vb$ must be relatively small, namely less than $\Theta(\epsilon)$, where $\epsilon$ is the target accuracy of the (overall) IPM solver. We emphasize that the error-adjustment vector $\vb \in \R^n$ is \textit{user-controlled}, as long as it satisfies the above conditions. Then, these conditions are \textit{sufficient} to guarantee that our predictor-corrector IPM (see Algorithm~\ref{algo:pcc} in Section~\ref{section:pcc}) converges to a solution $(\xopt, \yopt, \sopt)$ with a duality measure $\muopt < \epsilon$ in $\Ocal(\sqrt{n} \log \frac{\mu_0}{\epsilon})$ outer iterations. Importantly, the final solution is \textit{exactly} (and not approximately) feasible. 

A few additional remarks are necessary to better understand our results for the error-adjusted inexact predictor-corrector IPM.  First, we note that our method achieves the best-known outer iteration complexity for predictor-corrector IPMs.  Second, the error tolerance of the approximate linear equation solver does not directly depend on $n$ and is constant if $\epsilon$ is constant. Third, there are many potential constructions of $\vb$ which fulfill the above guarantees. This raises the problem of finding efficient constructions of $\vb$ for an inexact linear solver, as this has significant impact on the efficiency of the method. 

To address the last point, we exhibit efficient, practical methods to compute $\dytilde$ and $\vb$ by adapting the preconditioned conjugate gradient (PCG) algorithm for constraint matrices $\Ab$ that are short-and-fat, tall-and-thin, or even have exact low-rank (see Section~\ref{sxn:PCG_main_text} and Appendix~\ref{section:solver} for details). More precisely, we show that using PCG, we can compute $\dytilde$ and $\vb$ in $\Ocal\left(\log \frac{n\mu}{\epsilon}\right)$ inner iterations, where each inner iteration is simply a matrix-vector product. It is notable that this inner iteration complexity \textit{does not depend} on the spectrum or the condition number of the input matrix $\Ab\Db^2\Ab^T$. This is particularly important since the condition number of this matrix changes over iterations and might increase significantly as the outer iterations of the predictor-corrector IPM approach the optimal solution.

Our second contribution in this paper is a novel analysis of the classical, inexact predictor-corrector IPMs in the special setting where the final solution is allowed to be only \textit{approximately feasible}. More precisely, assume that $\dytilde$ is an approximate solution to the linear system of eqn.~(\ref{eq:intro_normal_eq}) that satisfies the following two conditions: 
\begin{flalign}\label{eq:solve_def}
    \|\dytilde-(\Ab\Db^2\Ab^T)^{-1}\pb\|_{\Ab\Db^2\Ab^T} \leq \delta
    ~~~\text{and}~~~
    \|\Ab\Db^2\Ab^T\dytilde - \pb\|_2 \leq \delta.
\end{flalign}
Here $\delta$ is the error tolerance of the solver and we note that the first condition guarantees that the exact and the approximate solutions are close with respect to the energy norm, while the second condition guarantees that the two-norm of the residual error of the solver is small. (See Section~\ref{section:notation_background} for notation.) We provide a novel analysis of the standard predictor-corrector method described in~\citet{wright1997primal} when approximate solvers that satisfy the above conditions are used. Theorem~\ref{thm:pcu_final} proves that these conditions suffice in order to prove that the standard predictor-corrector algorithm (see Algorithm~\ref{algo:pcu} in Appendix~\ref{section:pcu}) converges to a solution $(\xopt, \yopt, \sopt)$ with duality measure $\muopt$ such that $\|\Ab\xopt - \bb\|_2 < \epsilon$ and $\muopt < \epsilon$ in $\Ocal(\sqrt{n} \log \frac{\mu_0}{\epsilon})$ outer iterations. Assuming that $\mu_0$ and $\epsilon$ are constant, the accuracy parameter $\delta$ is set to $\Theta\left(\nicefrac{1}{\sqrt{n}} \right)$ at all iterations of the algorithm.

A few remarks are necessary to better understand the second result. First, the outer iteration complexity is essentially equivalent to the ``optimal'' $\Ocal(\sqrt{n} \log \frac{\mu_0}{\epsilon})$ iteration complexity of the exact predictor-corrector IPM methods and exhibits linear convergence in the accuracy parameter $\epsilon$. Second, the accuracy parameter $\delta$ for the approximate solver is, generally, proportional to $1/\sqrt{n}$, which is similar to the condition of~\citet{daitch2008faster}. It is worth noting that our proof collapses if the error bound exceeds this threshold, and an interesting open problem is whether this condition is necessary for predictor-corrector IPMs. Third, the final solution vector is only approximately (and not exactly) feasible, satisfying $\|\Ab\xopt - \bb\|_2 < \epsilon$ for the accuracy parameter\footnote{For notational simplicity, we use the same accuracy parameter $\epsilon$ for both the duality measure and the approximate feasibility of the final solution vector. Our analysis can be easily extended to use different accuracy parameters.} $\epsilon$. Fourth, for the same family of matrices as in our previous approach (tall-and-thin, short-and-fat, or exact low-rank $k \ll \min\{m,n\}$), we can again show that by using PCG solvers, we can efficiently compute an approximate solution $\dytilde$ in $\Ocal\left(\log \frac{{\small \sigma_{\max}(\Ab\Db)}\,n \mu}{\delta}\right)$ iterations of the preconditioned solver, where $\sigma_{\max}(\Ab\Db)$ is the largest singular value of the matrix $\Ab\Db$. We emphasize that, unlike our previous approach that uses the error-adjustment vector $\vb$, the convergence of the standard inexact predictor-corrector IPMs depends logarithmically on properties of the input matrix $\Ab\Db$ at each iteration.

Our two contributions exhibit a trade-off between algorithmic simplicity and theoretical guarantees. On one hand, the standard predictor-corrector IPM can be used without modifications with an iterative linear solver, but will not return a feasible solution and will need higher solver accuracy that depends on the largest singular value of $\Ab\Db$. Alternatively, the predictor-corrector method can be slightly modified to use an error-adjustment vector with the added benefits of obtaining an exactly feasible solution and removing dependence of the inner iteration complexity on the largest singular value of $\Ab\Db$.

We conclude by noting that our proof techniques are flexible and can be extended to analyze long-step and short-step IPMs, which are, however, less interesting in practice.

\subsection{Related Work}
\label{section:related_work}

Due to the central importance of linear programming in computer science, there exists a large body of work on LPs and IPMs specifically.  We refer the reader to the 2012 survey of \citet{gondzio2012interior} for more information on the broader state of IPMs, as we focus on literature that is most closely related to our work.  Recall that our main focus in this paper is a theoretical analysis of the outer iteration complexity of inexact  predictor-corrector IPMs with and without a correction vector that guarantees an exactly feasible solution.

Since the 1950s, there has been continual effort in the theoretical computer science community to develop new LP solvers with improved worst-case asymptotic time complexity.  Presently, there is no single fastest LP solver over all typical regimes of LPs.  The work of~\citet{lee2019solving} provides an IPM which requires $\Ocaltil(\sqrt{\rank(\Ab)} \log \frac{1}{\epsilon})$ outer iterations and $\Ocaltil(1)$ linear system solves at each outer iterations.  The recent works of \citet{cohen2021solving,song2021oblivious} have a total time complexity of $\Ocal^*(n^\omega \log \frac{n}{\epsilon})$ where $\omega \sim 2.37$ is the current best-known exponent of matrix multiplication. The work of~\citet{van2020solving} provides the theoretically fastest solver when $\Ab$ is tall and dense, in which case the outer iteration complexity is $\Ocaltil(\sqrt{n})$ and the total time complexity is $\Ocaltil(mn + n^3)$. All three of these works provide short-step IPMs, which have been found to converge slowly in practice.  Additionally, these works leverage techniques such as fast matrix multiplication and inverse maintenance, which, due to numerical instability and large constant factors, are generally ineffective in practice.  Our algorithms do not depend on any of these techniques. We instead focus on the predictor-corrector method, which is highly effective in practice, yet still has strong theoretical guarantees, with an outer iteration complexity of $\Ocal(\sqrt{n} \log \frac{\mu_0}{\epsilon})$ and one linear system solve per outer iteration. Our method can be fully implemented in less than 150 lines of code using well-established numerical techniques such as preconditioned conjugate gradient descent, as we show in Appendix \ref{section:solver}. The time complexity of this inexact linear system solver is $\Ocaltil(\nnz{\Ab} + k^3)$, where $k$ is the rank of $\Ab$.

In our work, we analyze the prototypical predictor-corrector algorithm described in~\cite{wright1997primal}.  One variant of this method, given by \citet{mehrotra1992implementation}, is considered the industry-standard approach to solving LPs and is perhaps the most common IPM used in linear programming packages~\cite{schork2020implementation,almeida2015convergence}. We do note that Mehotra's algorithm (unlike the standard predictor-corrector IPMs) does not come with provable accuracy guarantees. Development of new predictor-corrector variants along with theoretical analyses is ongoing.  Examples include the variant of Mehrotra's predictor-corrector IPM given by~\citet{salahi2008mehrotra} or the analysis by~\citet{almeida2015convergence} of a predictor-corrector method specifically suited for LPs arising in transportation problems \cite{bastos1993interior}. Other recent works includes the paper by~\citet{schork2020implementation} on empirically evaluating Mehrotra's algorithm when using preconditioned conjugate gradient descent to solve the normal equations at each step. The work of~\citet{yang2018arc} provides an infeasible predictor-corrector method with $\Ocal(n \log \frac{1}{\epsilon})$ outer iteration complexity and empirically demonstrates its competitiveness with existing methods. The importance of predictor-corrector methods motivates us to develop a better theoretical understanding of their convergence properties when using inexact linear solvers.

Multiple works have analyzed the impact of using inexact linear solvers within IPM algorithms, with early examples being~\cite{bellavia1998inexact} and~\cite{mizuno1999global}.  One method which is relevant to our work is that of~\citet{monteiro2003convergence} which guaranteed the convergence of a long-step IPM by correcting the error of the inexact solver using a correction vector $\vb$ as we describe in eqn.~(\ref{eq:solvev_guarantee}). This idea was further developed by~\citet{chowdhury2020speeding}, which introduced a more efficient construction of $\vb$. Another example of such works is~\citet{daitch2008faster}, which gives a short-step IPM alongside an inexact Laplacian system solver to solve the generalized max-flow problem. However, the analysis of their inexact short-step IPM does not seem to be directly applicable to solving general LPs where the constraint matrix is not Laplacian. We improve over these prior works by analyzing the inexact predictor-corrector method using two different approaches and we make minimal assumptions of the linear system solver and LP.

We further note that recently various first-order methods (with proper enhancements)  are also being explored to identify high-quality solutions to large-scale LPs quickly~\cite{basu2020eclipse, lin2021admm, applegate2021practical}. However, most of these endeavors are based on the combinations of  existing heuristics and do not come with theoretical guarantees.

%% file: 2_notation_background.tex
\section{Background}
\label{section:notation_background}

\subsection{Notation}

For any natural number $n$, let $[n] = \{1,2,...n\}$. Bold capital letters denote matrices (e.g, $\Ab$); bold lower case letters denote vectors (e.g., $\xb$); and the $i$-th element of vector $\xb$ is written as $\xb_i$. Let $\Ib_n$ denote the $n \times n$ identity matrix; let $\one_m$ and $\zero_m$ denote length $m$ vectors of all ones and zeroes respectively. We define the norm of a vector $\|\xb\|_p$ to be its well-known $\ell_p$ norm and the norm of a matrix $\|\Ab\|_p$ to be the induced $\ell_p$ norm, i.e. $\|\Ab\|_p = \max_{\|\xb\|_p = 1} \|\Ab\xb\|_p$.  We also use the \emph{energy norm} $\|\xb\|_\Mb = \sqrt{\xb^T \Mb \xb}$, where $\xb$ is a vector and $\Mb$ is a symmetric positive definite matrix. We denote the Hadamard (element-wise) product of two vectors $\ub,\vb$ as $\ub \circ \vb$.  Finally, we denote the Moore-Penrose pseudoinverse of a matrix $\Ab$ as $\Ab^\dagger$. 

\subsection{Background}

Interior point methods using an exact linear solver iteratively converge towards a primal-dual solution $(\xopt, \yopt, \sopt)$, which optimally solves the primal and dual LPs of eqns.~(\ref{eq:primal_lp_def}, \ref{eq:dual_lp_def}). The direction of each iterative step is determined by solving the so-called \textit{normal equations}:
\begin{subequations} \label{eq:normal_exact}
\begin{flalign}
    \Ab\Db^2\Ab^T\dyexact &= - \sigma \mu \Ab \Sb^{-1}\one_n + \Ab \xb, \\
    \dsexact &= - \Ab^T \dyexact,\label{eqn:1b} \\
    \dxexact &= -\xb + \sigma \mu \Sb^{-1}\one_m - \Db^2 \dsexact.
\end{flalign}
\end{subequations}
In the above, $\sigma \in [0,1]$ is the \emph{centering parameter}, which controls the tradeoff between progressing towards the optimal solution and staying near the central path. Let $\pb$ be equal to the right-hand-side of eqn.~(\ref{eq:normal_exact}a), i.e.,
\begin{equation}\label{eq:p_vec_def}
    \pb = - \sigma \mu \Ab \Sb^{-1}\one_n + \Ab \xb.
\end{equation}

Path-following IPM algorithms ensure that the iterates remain sufficiently far from the boundary of the convex polytope representing the feasible set of the primal and dual LPs and near the central path. In this paper, we use the $\ell_2$ neighborhood $\Ncal_2(\theta)$ defined as follows.
\begin{align}
    \Ncal_2(\theta) = \big\{ &(\xb,\yb,\sbb)\in \R^{2n+m}: \|\xb \circ \sbb - \mu \one_n\|_2 \leq \theta\mu,
    ~(\xb, \sbb) > 0\big\} \label{eq:neighborhood_2}.
\end{align}
The step size of the outer iterations will need to be dynamically determined to ensure that the iterates remain in the appropriate neighborhood. The following notation compactly describes the next iterate after a step of size $\alpha$:
\begin{flalign*}
    \xb(\alpha) = \xb + \alpha \dxexact,~ 
    \yb(\alpha) = \yb + \alpha \dyexact,~ 
    \sbb(\alpha) = \sbb + \alpha \dsexact,
    \text{ and }
    \mu(\alpha) = (\xb+\alpha\dxexact)^T(\sbb+\alpha\dsexact)/n.
\end{flalign*}
The following identities hold for the exact steps determined by eqn.~(\ref{eq:normal_exact}) above; see~\cite{wright1997primal} for details:
\begin{gather}
    \dxexact^T\dsexact 
     = \zero, \label{eq:exact_cross_zero} \\
    \sbb^T \dxexact + \xb^T \dsexact = -n\mu + n \sigma \mu \label{eq:exact_sx_xs}, \\
    \mu(\alpha) = (1-\alpha+\alpha \sigma)\mu. \label{eq:exact_mu_alpha_bound}  
\end{gather}

We will also use the following bound on $ \|\dxexact \circ \dsexact\|_2$ to argue that the iterative steps remain in the neighborhood defined by eqn.~(\ref{eq:neighborhood_2}); see Lemma 5.4 in~\citep{wright1997primal} for a proof:
\begin{flalign}
     (\xb, \yb, \sbb)\in\Ncal_2(\theta) &\Rightarrow \|\dxexact \circ \dsexact\|_2
    \leq \frac{\theta^2+n(1-\sigma)^2}{2^{3/2} (1-\theta)} \mu. \label{eq:l2_exact_cross_bound} 
\end{flalign}

%% file: 3_overview.tex
\section{Overview of our approach and proofs}\label{section:overview}

Inexactly solving the normal equations when determining the step direction in predictor-corrector IPMs adds new difficulties to the convergence analysis such methods. Identities that are critical in analyzing the the exact methods, such as $\dxexact^T\dsexact=0$, no longer hold. Another source of difficulty is that, even when starting from a feasible initial point, the iterates will become infeasible due to the error incurred by the solvers. We can handle this infeasibility in two different ways. First, we can assume bounds on the maximum error of the solver at each step, which would guarantee that the final solution is $\epsilon$-feasible, i.e. $\|\Ab\xopt - \bb\|_2 \leq \epsilon$. The second approach, previously introduced in~\cite{monteiro2003convergence}, is to adjust the error in each step to ensure that the next iterate is feasible. We analyze both approaches in our work. 

The following equation block designates the step of inexactly solving the normal equations, where the vector $\fb$ is the error residual incurred by solving for $\dytilde$ using an inexact linear solver:
\begin{subequations}\label{eq:normal_uncorrected}
\begin{flalign}
    \Ab\Db^2\Ab^T\dytilde &= - \sigma \mu \Ab \Sb^{-1}\one_n + \Ab\xb - \fb, \\
    \dstilde &= - \Ab^T \dytilde, \\
    \dxtilde &= -\xb + \sigma \mu \Sb^{-1}\one_n - \Db^2 \dstilde.
\end{flalign}
\end{subequations}
The next equation block deals with the case of an error-adjusted approximate step. The idea behind this error-adjustment is the construction of a vector $\vb$ with small norm such that the error vector $\fb$ is exactly equal to $-\Ab\Sb^{-1}\vb$. We then correct the primal step $\dxtilde$ by subtracting $\Sb^{-1}\vb$, which guarantees that $\Ab \dxtilde$ is equal to zero:
\begin{subequations} \label{eq:normal_corrected}
\begin{flalign}
    \Ab\Db^2\Ab^T\dytilde &= - \sigma \mu \Ab \Sb^{-1}\one_n + \Ab\xb + \Ab\Sb^{-1}\vb, \\
    \dstilde &= - \Ab^T \dytilde, \\
    \dxtilde &= -\xb + \sigma \mu \Sb^{-1}\one_n - \Db^2 \dstilde - \Sb^{-1}\vb .
\end{flalign}
\end{subequations}
Note that both the inexact and error-adjusted normal equations maintain dual feasibility of the iterate when starting from any dual feasible starting point, i.e., $\Ab^T\yb + \sbb = \cbb$.

In order to derive our theoretical bounds, we first analyzed the \textit{uncorrected} inexact predictor-corrector IPM, which is the original predictor-corrector IPM using an approximate solver denoted $\solve$. In the interest of space, we delegate the presentation and analysis of this algorithm to the Appendix (see Appendix~\ref{section:pcu} and Algorithm~\ref{algo:pcu}). The uncorrected inexact predictor-corrector IPM takes two steps (a predictor step and a corrector step) in each outer iteration. Starting from a point in $\Ncal_2(0.25)$, the algorithm takes a \textit{predictor step} with centering parameter $\sigma=0$ and a dynamically chosen step size $\alpha$, such that the iterate remains in $\Ncal_2(0.5)$.  The algorithm then takes a \textit{corrector step} with centering parameter $\sigma = 1$ and step size $\alpha = 1$, which returns the iterate back to $\Ncal_2(0.25)$.  The predictor step results in a multiplicative decrease in the duality measure, and the corrector step sets up the next predictor step, while only resulting in a slight additive increase in the duality measure. Our main result for Algorithm \ref{algo:pcu} in Appendix~\ref{section:pcu} is given by the following theorem.
\begin{theorem} \label{thm:pcu_final}
    Let $\epsilon > 0$ be a tolerance parameter and $(\xb_0, \yb_0, \sbb_0) \in \Ncal_2(0.25)$ with duality measure $\mu_0$ be a feasible starting point. Then, Algorithm~\ref{algo:pcu} (see Appendix~\ref{section:pcu}) converges to a dual-feasible point $(\xtopt, \ytopt, \stopt)$ with duality measure $\muopt$ such that $\|\Ab\xtopt - \bb\|_2 < \epsilon$ and $\muopt < 2\epsilon$ in $\Ocal(\sqrt{n} \log \frac{\mu_0}{\epsilon})$ outer iterations, where the approximate linear solver $\solve$ is called twice in each outer iteration with error tolerance parameter $\tolSolve = \min\{\frac{\sqrt{\epsilon}}{2^6},  \frac{\epsilon C_0}{2 \sqrt{n} \log \mu_0/\epsilon} \} $ and the constant $C_0$ is defined in Lemma \ref{lemma:pcu_predictor_step}.
\end{theorem}
Next, in Section \ref{section:pcc}, we present and analyze \textit{our main contribution}, the error-adjusted predictor-corrector method (Algorithm~\ref{algo:pcc}), which uses a linear solver\footnote{$(\dytilde,\vb) = \solvev(\Ab,\pb,\delta)$ takes three inputs: the input matrix $\Ab$, the response vector $\pb$, and the target accuracy (or tolerance) $\delta$ and returns an approximate solution $\dytilde$ and the error-adjustment vector $\vb$.} $\solvev$ to compute an approximate solution $\dytilde$ and an error-adjustment vector $\vb$ that satisfy eqns.~(\ref{eq:normal_corrected}). The main result of Section~\ref{section:pcc} is the following theorem.
\begin{theorem} \label{thm:pcc_final}
    Let $\epsilon > 0$ be a tolerance parameter and $(\xb_0, \yb_0, \sbb_0) \in \Ncal_2(0.25)$ with duality measure $\mu_0$ be a feasible starting point. Then, Algorithm \ref{algo:pcc} converges to a primal-dual feasible point $(\xtopt, \ytopt, \stopt)$ with duality measure $\muopt$ such that $\muopt < 2\epsilon$ in $\Ocal(\sqrt{n} \log \frac{\mu_0}{\epsilon})$ outer iterations, where $\solvev$ is called twice in each outer iteration with error tolerance $\nicefrac{\epsilon}{2^7}$.
\end{theorem}
The convergence analysis of both inexact predictor-corrector algorithms shares the same overall structure, which we now outline. First, we upper bound $\|\dxtilde \circ \dstilde\|_2$, a technical result that will be needed in upcoming steps. Second, we derive a bound for the left-hand side of the $\Ncal_2$ neighborhood condition $(\|\xb \circ \sbb - \mu \one_n\|_2)$ after step size $\alpha$. This bound depends on $\|\dxtilde \circ \dstilde\|_2$ and the error of the linear solver. Third, we find a value for the step size $\alpha$ that depends on $\|\dxtilde \circ \dstilde\|_2$, which keeps the next iterate in the appropriate neighborhood, $\Ncal_2(0.5)$. Fourth, we lower bound the step size $\alpha$ using the upper bound on $\|\dxtilde \circ \dstilde\|_2$. Fifth, we use the lower bound on $\alpha$ from the previous step to lower bound the multiplicative decrease in the duality measure after the predictor step. Finally, we prove that the corrector step with step size $\alpha=1$ returns the iterate to $\Ncal_2(0.25)$ by using the inequality from the second step and then bound the resulting additive increase in the duality measure.

For both predictor-corrector algorithms, the above structure provides a guaranteed decrease in the duality measure over a single step of the form $\mutilde_{1} \leq \left(1 - \frac{C_0}{\sqrt{n}} \right) \mu_0 + C_1\frac{\tolSolve}{n}$, where $C_0 \in (0,1)$, $C_1 \in [0,C_0/\sqrt{n}]$, and $\tolSolve > 0$ is the tolerance parameter for the corresponding linear solver.  We can use this relation to conclude (using standard arguments) that each algorithm converges to a point $(\xopt, \yopt, \sopt)$ with duality measure $\muopt < 2\epsilon$. 
We note that the proof of the inexact predictor-corrector IPM without using error-adjustment is simpler, partly because the duality measure during its inexact predictor-corrector step is always higher than the duality measure during the exact step. This is not the case for the inexact predictor-corrector IPM with error-adjustment, which needs extra care in bounding the duality gap decrease in each iteration. 

In Section~\ref{sxn:PCG_main_text} (see also Appendix~\ref{section:solver}), we show how the approximate linear solver $\solvev$ can be efficiently instantiated when the constraint matrix $\Ab$ has exact low rank (which includes as special cases tall-and-thin and short-and-fat matrices), by  using a preconditioned conjugate gradient (PCG) method.

%% file: 5_pcc.tex
\section{Error-adjusted Inexact Predictor-Corrector IPMs}
\label{section:pcc}

In this section, we introduce an algorithm that we will call error-adjusted inexact predictor-corrector IPM (Algorithm~\ref{algo:pcc}). This algorithm uses an inexact linear solver $\solvev$, which returns an approximate solution $\dytilde$ and a correction vector $\vb$ that satisfies the conditions of eqn.~(\ref{eq:solvev_guarantee}). This correction vector guarantees that the final solution will be exactly feasible and, as discussed in Section~\ref{section:intro}, Algorithm~\ref{algo:pcc} can tolerate larger errors for the inexact solver. We follow the proof sketch of Section~\ref{section:overview} to prove convergence guarantees and time complexity for Algorithm~\ref{algo:pcc}. We will assume that the matrix $\Ab$ has full row rank, ie., $rank(\Ab) = m \leq n$; see Appendix~\ref{section:low_rank} for extensions alleviating this constraint.

\begin{algorithm}[h] 
	\caption{Error-adjusted Inexact Predictor-corrector} \label{algo:pcc}%
	\begin{algorithmic}
		
		\State \textbf{Input:}
		$A\in\R^{m \times n}$, initial feasible point $(\xb^{0},\yb^{0},\sbb^{0}) \in \Ncal_2(0.25)$;  IPM tolerance $\epsilon> 0$.
		
		\vspace{1mm}
		\State \textbf{Initialize:} ~$k\gets 0$;
		
		\vspace{1mm}
		\While{~$\mu_k > 2\epsilon$} 

		\vspace{1mm}
		\textit{Predictor Step $(\sigma = 0)$:}
		\State(a) Compute $(\dytilde, \vb) =~\solvev(\Ab\Db^2\Ab^T, \Ab\xb, \nicefrac{\epsilon}{2^7})$.
		\vspace{1mm}
		\State (b) Compute $\dxtilde$ and $\dstilde$ via eqn.~(\ref{eq:normal_corrected}).
		\vspace{1mm}
		\State (c) Set  $\alpha = \min\left\{ \nicefrac{1}{2},~
        \left(\nicefrac{\mu}{16 \|\dxtilde \circ \dstilde\|_2} \right)^{1/2} \right\}$
		\vspace{1mm}
		\State (d)~Compute $(\xb_k,\yb_k, \sbb_k) = (\xb_k,\yb_k,\sbb_k) + \alpha (\dxtilde_k,\dytilde_k,\dstilde_k)$.
		\vspace{1mm}
		
		\textit{Corrector Step $(\alpha = 1, \sigma = 1)$:}
		\State(e) Compute $(\dytilde, \vb) = ~\solvev(\Ab\Db^2\Ab^T, -\mu \Ab\Sb^{-1}\one_n + \Ab\xb, \nicefrac{\epsilon}{2^7})$.
		\vspace{1mm}
		\State (f) Compute $\dxtilde$ and $\dstilde$ via eqn.~(\ref{eq:normal_corrected}).
		\vspace{1mm}
		\State (g)~Compute $(\xb_{k+1},\yb_{k+1}, \sbb_{k+1}) = (\xb_k,\yb_k,\sbb_k) + (\dxtilde_k,\dytilde_k,\dstilde_k)$. 
		\vspace{1mm}

		\State (h) $k \gets k + 1$.
		
		\EndWhile
	\end{algorithmic}
\end{algorithm}
We proceed by expressing the difference of the exact vs. approximate solutions, using eqn.~(\ref{eq:normal_exact}) vs. eqn.~(\ref{eq:normal_corrected}):
\begin{align}
    \dyexact - \dytilde &= -(\Ab\Db^2\Ab^T)^{-1} \Ab\Sb^{-1} \vb, \label{eq:pcc_y_diff}\\
    \dsexact - \dstilde &= -\Ab^T(\dyexact - \dytilde) \nonumber  \\
    &= \Ab^T(\Ab\Db^2\Ab^T)^{-1} \Ab\Sb^{-1} \vb, \label{eq:pcc_s_diff}\\
    \dxexact - \dxtilde &= -\Db^2(\dsexact - \dstilde) + \Sb^{-1}\vb \nonumber
    \\
    &= -\Db^2\Ab^T(\Ab\Db^2\Ab^T)^{-1} \Ab\Sb^{-1} \vb + \Sb^{-1}\vb.  \label{eq:pcc_x_diff}
\end{align}
We prove that Algorithm \ref{algo:pcc} converges to a point $(\xtopt, \ytopt, \stopt)$, such that $\mutopt < 2\epsilon$, $\Ab\xtopt = \bb$, and $\Ab^T\ytopt + \stopt = \cbb$ in $\Ocal(\sqrt{n} \log \nicefrac{\mu_0}{\epsilon})$ outer iterations. First, we start with a technical result to bound $\|\dxtilde \circ \dstilde\|_2$. (All proofs are delegated to Appendix~\ref{section:pcc_proofs}.)
\begin{lemma}
    Let $(\xb,\yb,\sbb) \in \Ncal_2(\theta)$ and let $(\dxtilde, \dytilde, \dstilde)$ be the step calculated from the inexact normal equations without error-adjustment (see eqn.~(\ref{eq:normal_corrected})). Then,
    \begin{align*}
        \|\dxtilde \circ \dstilde\|_2 \leq \frac{\theta^2+n(1-\sigma)^2}{2^{3/2} (1-\theta)} \mu 
         +3 \sqrt{\frac{(\theta^2+n(1-\sigma)^2) \mu}{(1-\theta)}} \|(\Xb\Sb)^{-1/2}\vb\|_2 
        + 2\|(\Xb\Sb)^{-1/2}\vb\|_2^2.
    \end{align*}
\end{lemma}
This inequality represents a key technical contribution of our results.  The proofs of \citet{monteiro2003convergence, chowdhury2020speeding} on the convergence of a long-step IPM cannot be readily extended to the predictor-corrector method, as the predictor-corrector algorithm requires finer control over deviations from the central path due to using the $\ell_2$-neighborhood.  However, this more restrictive neighborhood allows it to achieve better outer iteration complexity.  Observe that in the corrector step, when $\sigma = 1$, our bound does not directly scale with $n$ in this case, in contrast to Lemma 16 in \cite{chowdhury2020speeding} and Lemma 3.7 in \cite{monteiro2003convergence}.

We can use the previous inequality to bound the deviation of the iterate from the central path after a step of size $\alpha$.
\begin{lemma}
If $\alpha \in [0,1]$, then
\begin{align*}
    \|\xtilde(\alpha) \circ \stilde(\alpha) - \mutilde(\alpha) \one_n\|_2
    \leq (1-\alpha)\|\xb \circ \sbb - \mu \one_n\|_2 
    + \alpha^2\|\dxtilde \circ \dstilde\|_2 + 2\alpha \|\vb\|_2.
\end{align*} 
\end{lemma}

Given the previous bound, we can then derive a step size $\alpha$ which guarantees that the iterate remains in $\Ncal_2(0.5)$ after the predictor step.
\begin{lemma}
   If $(\xb,\yb,\sbb) \in \Ncal_2(0.25)$,
        $\alpha = \min\left\{ \nicefrac{1}{2},
        \left(\nicefrac{\mu}{16 \|\dxtilde \circ \dstilde\|_2} \right)^{1/2} \right\}$, and
    $\|\vb\|_2 \leq \nicefrac{\mu}{32}$,
    then the predictor step  $(\xtilde(\alpha), \ytilde(\alpha), \stilde(\alpha)) \in \Ncal_2(0.5)$.
\end{lemma}

We then show that the predictor step with step size $\alpha$ as given in the above lemma guarantees a multiplicative decrease in the duality gap. Recall that $\sigma = 0$ in the predictor step when solving the normal equations.
\begin{lemma} 
If $(\xb,\yb,\sbb) \in \Ncal_2(0.25)$,
     $\alpha = \min\left\{ \nicefrac{1}{2},
        \left(\nicefrac{\mu}{16 \|\dxtilde \circ \dstilde\|_2} \right)^{1/2} \right\}$, and $\|\vb\|_2 \leq \nicefrac{\mu}{32}$,
then the predictor step $(\xtilde(\alpha), \ytilde(\alpha), \stilde(\alpha))$ remains in $\Ncal_2(0.5)$ and there exists a constant $C_0 \in (0,1)$ such that, 
\begin{equation*}
    \frac{\mutilde(\alpha)}{\mu} \leq 1 - \frac{C_0}{\sqrt{n}}.
\end{equation*}
\end{lemma}
After the previous lemma, we have shown that the predictor step results in a multiplicative decrease in the duality gap, while keeping the next iterate in the neighborhood $\Ncal_2(0.5)$.  We then show that the corrector step returns the iterate to the $\Ncal_2(0.25)$ neighborhood, while increasing the duality gap by a small additive amount. 
\begin{lemma}
Let $(\xb, \yb, \sbb) \in \Ncal_2(0.5)$ and $\|\vb\|_2 \leq \nicefrac{\mu}{2^7}$. Then, the corrector step $(\xtilde(1), \ytilde(1), \stilde(1)) \in \Ncal_2(0.25)$ and $|\mutilde(1) - \mu| \leq \frac{1}{\sqrt{n}}\|\vb\|_2$.
\end{lemma}
Overall, the structure of the proof approach for the inexact predictor-corrector method (Appendix \ref{section:pcu}) is similar to the proof structure shown here.  However, proving the individual lemmas for the error-adjusted algorithm requires slightly more work, since the correction step adds an additional adjustment to the iterates in each step, which must be accounted for. 

In comparison to the proof of the standard predictor-corrector method found in \citet{wright1997primal}, the general idea and organization of the lemmas are shared, however, accounting for the error in each step can lead to unwieldy and complicated formulas.  An important part of our proof that partially alleviates this problem is to generalize the inequalities used in \citet{wright1997primal} to depend smoothly on the error of the inexact linear system solve, so we recover the Statement of each lemma when the error is zero.  This allows us to appropriately set the linear system solver precision to give up small factors in the tightness of our result in each lemma, which can then be offset by increasing the precision of the linear system solver.  By doing so, we ensure that the added proof complexity for the inexact predictor-corrector method is locally resolved in each lemma. As a result, our proofs are conceptually simpler than prior work proving the convergence of inexact IPMs, such as \citet{chowdhury2020speeding} and \citet{daitch2008faster}.

\subsection{Implementing $\solvev$}\label{sxn:PCG_main_text}

We demonstrate that $\solvev$ can be effectively implemented using a preconditioned conjugate gradient (PCG) method that also constructs the correction vector $\vb$ that satisfies the conditions of eqn.~(\ref{eq:solvev_guarantee}). By employing a randomized preconditioner, the resulting PCG method guarantees an exponential decrease in the energy norm of the residual. Here we sketch our approach, and we provide details in Appendix~\ref{section:solver}.

Let $\Ab\Db=\Ub\Sigmab\Vb^\ts$ be the thin SVD representation with $\Vb \in \mathbb{R}^{m \times n}$ and $\Wb \in \R^{n \times w}$ be an oblivious sparse sketching matrix which satisfies, for some accuracy parameter $\zeta \in (0,1/2)$:
\begin{equation}
    \|\Vb\Wb\Wb^T\Vb^T - \Ib_m\|_2 \leq \frac{\zeta}{2},
\end{equation}
with probability at least $1-\eta$. The work of \citet{cohen2015optimal} shows how to construct such a matrix $\Wb$ fulfilling this guarantee with sketch size $w = \Ocal(\nicefrac{m}{\zeta^2}\cdot\log \nicefrac{m}{\eta})$ and $\Ocal(\nicefrac{1}{\zeta}\cdot\log \nicefrac{m}{\eta})$ non-zero entries per row. Next, we use the above sketching matrix to define $$\Qb = \Ab\Db\Wb\Wb^T\Db\Ab^T.$$ We note that $\Qb$ does not need to be explicitly constructed, since we will only use the inverse of its square root $\Qb^{-1/2}$ (see Algorithm~\ref{algo:PCG} in Appendix~\ref{section:solver} for details).
More specifically, since $\Wb$ has $\log \nicefrac{m}{\eta}$ non-zero entries per row and $\Db$ is a diagonal matrix, $\Ab\Db\Wb$ can be computed in $\Ocal(\nnz{\Ab}\cdot  \log \nicefrac{m}{\eta})$ time. Then, computing $\Qb^{-1/2}$ via the SVD of $\Ab\Db\Wb$ takes $\Ocal(m^3 \log \nicefrac{m}{\eta})$ time. The overall time complexity to compute $\Qb^{-1/2}$ is $\Ocal(\nnz{\Ab} \cdot \log \nicefrac{m}{\eta} + m^3 \log \nicefrac{m}{\eta})$.

\begin{figure}[h]
\begin{minipage}{0.5\textwidth}
    \centering
    \includegraphics[width=\textwidth]{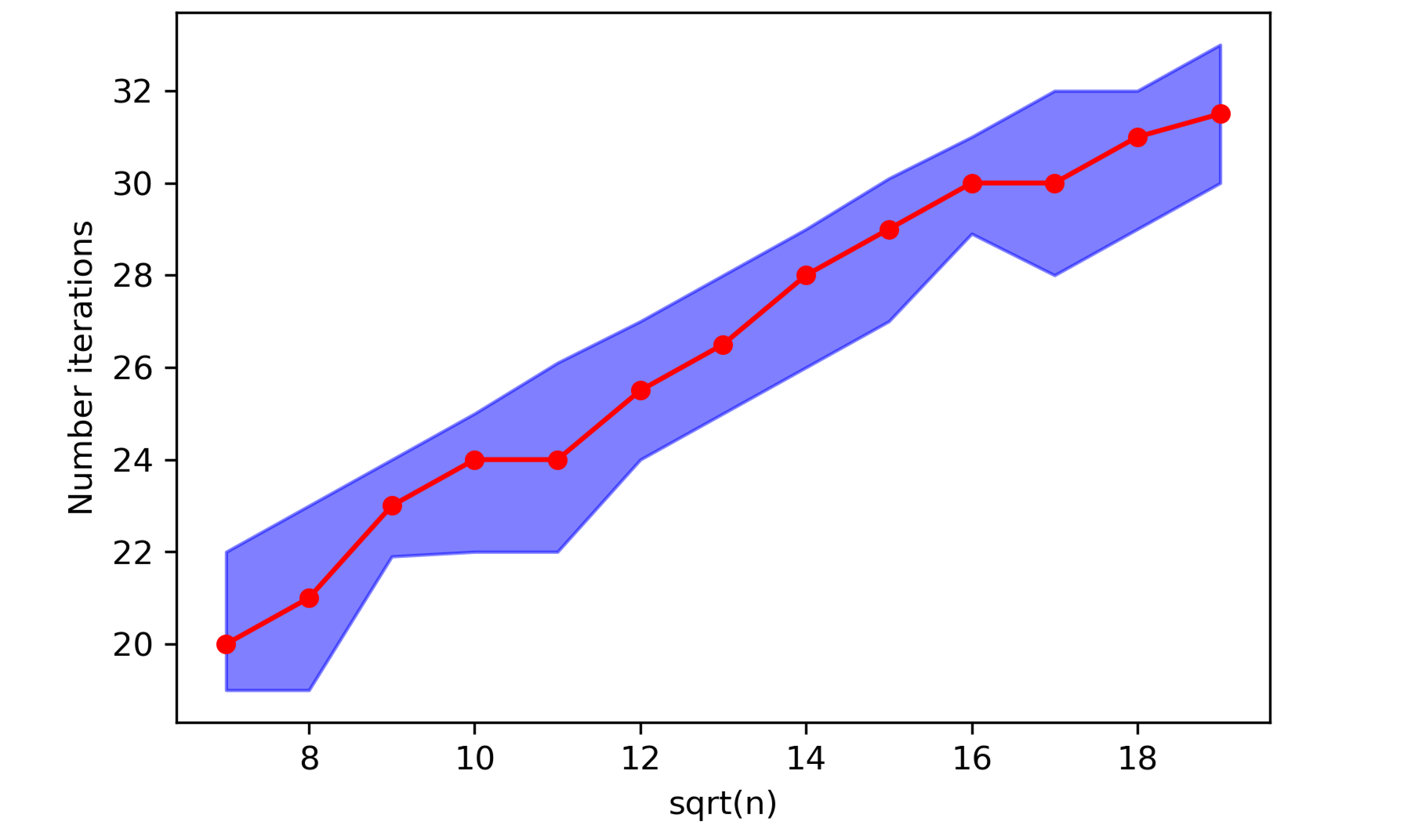}
    \caption{This graph demonstrates the linear relationship between the number of iterations and $\sqrt{n}$, as predicted by Theorem~\ref{thm:pcc_final}.  The line shows the median number of iterations and the intervals designate the 10\% and 90\% quantiles out of 60 repetitions. Other parameters are $m = 20$; $\epsilon = 0.1$; and solver tolerance $0.001$.} 
    \label{fig:num_it_vs_n}
\end{minipage}%
\hspace{\columnsep}
\begin{minipage}{0.5\textwidth}
    \centering
    \centering
    \includegraphics[width=\textwidth]{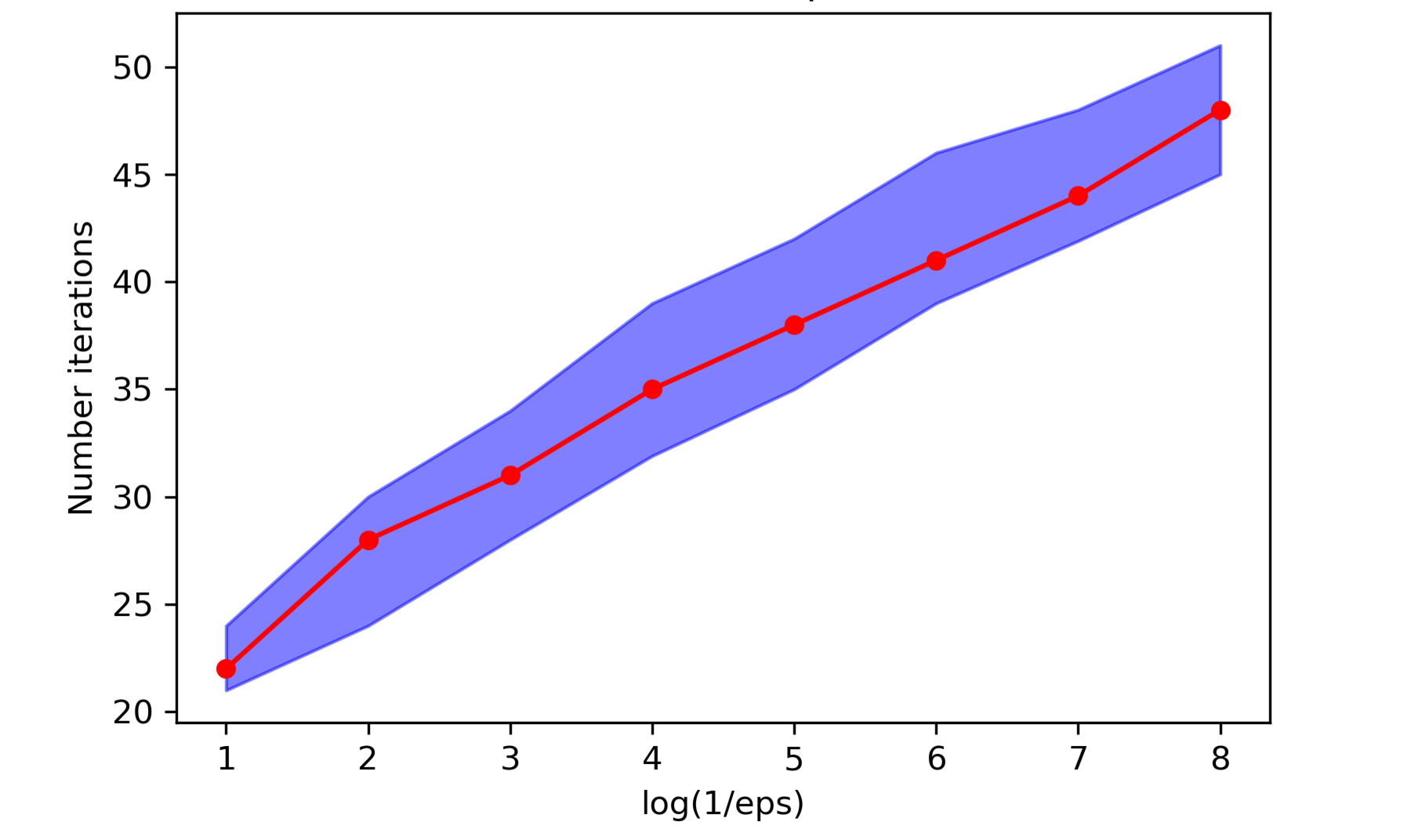}
    \caption{This graph demonstrates the linear relationship between the number of iterations and $\log(1/\epsilon)$, as predicted by Theorem~\ref{thm:pcc_final}.  The line shows the median number of iterations and the confidence intervals designate the 10\% and 90\% quantiles out of 60 repetitions. Other parameters are $m = 30$; $n = 70$; and solver tolerance equal to  $\epsilon$.}
    \label{fig:num_it_vs_eps}
\end{minipage}
\end{figure}

Next, we prove that the vector $\ztilde^t$ returned by PCG (Algorithm~\ref{algo:PCG}, Appendix~\ref{section:solver}) fulfills the following inequality with probability at least $1-\eta$:
\begin{gather*}
    \|\Qb^{-1/2} (\Ab\Db^2\Ab^T) \Qb^{-1/2} \ztilde^t -     \Qb^{-1/2}\pb\|_2 \leq \zeta^t \|\Qb^{-1/2} \pb\|_2,
\end{gather*}
for the aforementioned error-parameter $\zeta$.    
Given $\Wb$ (the sketching matrix used to construct the preconditioner), we proceed to construct the error-adjustment vector $\vb$ as follows: 
\begin{gather*}
    \vb = (\Xb\Sb)^{1/2} \Wb (\Ab\Db\Wb)^\dagger (\Ab\Db^2 \Ab^T \dytilde - \pb),
\end{gather*}
where $\dytilde = \ztilde^t$ after $t = \Ocal \left(\log \frac{n\mu}{\epsilon} \right)$ iterations. The additional time needed to compute the error-adjustment vector $\vb$ is negligible, since it only adds matrix vector products, using quantities that have already been computed and are available to the algorithm. Combining randomized preconditioning, conjugate gradients, and our proposed construction of the error-adjustment vector $\vb$ is theoretically and practically efficient for short-and-fat, tall-and-thin, and exact low-rank matrices. It can also take advantage of any sparsity in the input matrix, since both our preconditioner construction and conjugate gradient methods leverage input sparsity.

%% file: main_experiments.tex
\subsection{Empirical validation of Theorem \ref{thm:pcc_final}}

We experimentally validate the predictions of Theorem~\ref{thm:pcc_final} on synthetic data.  Specifically, we observe the predicted linear relationship between the number of iterations vs. $\sqrt{n}$ and $\log(1/\epsilon)$, when the precision of $\solvev$ is set to $\Ocal(\epsilon)$. To generate the synthetic LPs, we sample a constraint matrix $\Ab \in \R^{m \times n}$, the initial primal variable $\xb_0 \in \R^n$, and the initial dual variable $\yb_0 \in \R^m$, where each entry is sampled uniformly over an appropriate interval. From these points, we can choose an initial slack variable such that $\|\xb_0 \circ \sbb_0 - \mu_0\one_n\|_2 < (0.25) \mu_0$.  The initial primal-dual point along with the constraint matrix completely describes the LP, since the initial point is assumed to be primal-dual feasible. We implement $\solvev$ to find the primal-dual step and correction vector $\vb$ by uniformly sampling a vector $\vb \in \R^n$ fulfilling $\|\vb\|_2 = \Ocal(\epsilon)$ and then solving for the corresponding step in eqn.~(\ref{eq:normal_corrected}a) exactly; see Figures~\ref{fig:num_it_vs_n} and~\ref{fig:num_it_vs_eps}. We also test our preconditioned gradient descent method for finding the primal-dual step, and error-adjustment vector $\vb$ in Appendix~\ref{section:experiments_appendix} and we find that the performance of our approach is comparable to the exact method shown here.

%% file: conclusions.tex
\section{Conclusions}
We present and analyze an inexact predictor-corrector IPM algorithm that uses preconditioned inexact solvers to accelerate each iteration of the IPM, without increasing the number of iterations or sacrificing the feasibility of the returned solution. In future work, it is of interest to extend this framework to design fast and scalable algorithms for  more general convex problems, such as semidefinite programming.    

%% file: appendix_a.tex
\section{Additional Proofs}

\begin{lemma} \label{lemma:ipm_convergence}
If the duality measure of the IPM decreases with the relation,
$$\mu_{1} \leq \left(1 - \frac{C_0}{\sqrt{n}}\right)\mu_0 + C_1 \epsilon
~~~\text{for}~C_0 \in (0,1),~~C_1 \in [0,\frac{C_0}{\sqrt{n}}),$$

for all $\mu_0 \geq 2\epsilon$, then the IPM algorithm converges to a point $(\xtopt, \ytopt, \stopt)$ with duality measure $\mutopt < 2\epsilon$ in $\frac{\sqrt{n}}{C_0} \log \frac{\mu_0}{\epsilon}$ outer iterations. 
\end{lemma}
\begin{proof}

Each algorithm terminates when $\mu_k < 2\epsilon$ indicating convergence has been reached.  Therefore, we assume that at each iteration $\mu_k \geq 2\epsilon$. 

Although the starting point is assumed to be feasible.  This fact can be ignored in the convergence analysis.  The constraint matrix of the linear program is assumed to be full rank and the system is undetermined.  This implies that for all $(\xb_0, \yb_0, \sbb_0)$ there exists a vector $\bb$ so that the starting point is feasible.  Since the duality measure does not depend on $\bb$, the decrease in the duality gap occurs whether or not the starting point is feasible.

Given this, we can define a recurrence relation $\Tcal(k)$ such that $\mu_k \leq \Tcal(k)$, where,
$$
\Tcal(k) = \left(1 - \frac{C_0}{n^p}\right)\Tcal(k-1) + C_1 \epsilon
~~~~~~
\Tcal(0) = \mu_0.
$$

If we define $\xi = \left(1 - \frac{C_0}{\sqrt{n}}\right)$, then we have a recurrence relation of the form $\Tcal(k) = \xi\Tcal(k-1) + C_1 \epsilon$. The solution to the recurrence relation is.
\begin{flalign*}
    \Tcal(k) &= \frac{\epsilon C_1 (1 - \xi^k)}{1 - \xi} + \xi^k \mu_0 \\
    &\leq \frac{\epsilon C_1}{\nicefrac{C_0}{n^p}} + \xi^k \mu_0 \\
    &\leq \epsilon + \xi^k \mu_0.
\end{flalign*}

Therefore, $\mu_k \leq T(k) \leq 2 \epsilon$ if $\xi^k\mu_0 \leq \epsilon$. We can prove the outer iteration complexity using a standard argument by substituting back in $\left(1 - \frac{C_0}{\sqrt{n}}\right)$ for $\xi$ and using the identity $\log(1+\beta) \leq \beta$ for all $\beta>-1$. We prove that $\mu_k < 2\epsilon$ if $k \geq \frac{\sqrt{n}}{C_0} \log \frac{\mu_0}{\epsilon}$.
\begin{gather*}
    k \geq \frac{\sqrt{n}}{C_0} \log \frac{\mu_0}{\epsilon} \\
    \Rightarrow \frac{-k C_0}{\sqrt{n}} \leq \log \frac{\epsilon}{\mu_0} \\
    \Rightarrow k \log \left(1 - \frac{C_0}{\sqrt{n}}\right) \leq \log \frac{\epsilon}{\mu_0} \\
    \left(1 - \frac{C_0}{\sqrt{n}}\right)^k \mu_0 \leq \epsilon 
\end{gather*}
Therefore, the IPM algorithm is guaranteed to converge in $\frac{\sqrt{n}}{C_0} \log \frac{\mu_0}{\epsilon}$ outer iterations.
\end{proof}

\begin{lemma} \label{lemma:hadamard_cauchy}
    Let $\ub,\vb \in \R^n$.
    $$
    \|\ub \circ \vb\|_2 \leq \|\ub\|_2 \|\vb\|_2
    $$
\end{lemma}
    
\begin{proof}

\begin{flalign*}
    \|\ub \circ \vb\|_2^2 =\sum_{i=1}^n (\ub_i\vb_i)^2
    = \sum_{i=1}^n \ub_i^2\vb_i^2
    \leq \left(\sum_{i=1}^n \ub_i^2 \right)\left(\sum_{i=1}^n \vb_i^2 \right)
    = \|\ub\|_2^2\|\vb\|_2^2.
\end{flalign*}

\end{proof}

\begin{lemma} \label{lemma:lower_singular_value_bound}
If $\Mb$ is an $m \times n$ matrix of full row rank and $\xb$ is an arbitrary vector in the row space of $\Mb$, then $\|\Mb\xb\|_2 \geq \sigma_m(\Mb)\|\xb\|_2$.
\end{lemma}   
\begin{proof}
Let $\Mb = \Ub \Sigmab \Vb^T$ where $\Ub,\Sigmab \in \R^{m\times m}$ and $\Vb \in \R^{n\times m}$ constitute the thin-SVD of $\Mb$.  Since $\Ub$ is an orthonormal matrix, multiplying a vector by $\Ub$ does not change the $\ell_2$-norm.  Therefore, we have the following with $\yb = \Vb^T \xb$:
\begin{flalign*}
    \|\Mb\xb\|_2^2&= \|\Ub \Sigmab \Vb^T \xb\|_2^2= \|\Sigmab \Vb^T \xb\|_2^2\\
    &= \|\Sigmab \yb\|_2^2= \sum_{i=1}^m \sigma_i(\Mb)^2 \yb_i^2 \\
    &\geq \sigma_m(\Mb)^2 \sum_{i=1}^m \yb_i^2\geq \sigma_m(\Mb)^2 \|\yb\|_2^2,
\end{flalign*}
The vectors $\yb$ and $\xb$ have the same norm since $\yb^T\yb = \xb^T \Vb \Vb^T \xb = \xb^T\xb$, and the columns of $\Vb$ are orthonormal. Therefore, $\|\Mb\xb\|_2^2 \geq \sigma_m(\Mb)^2\|\xb\|_2^2$. 

\end{proof}

\begin{lemma} \label{lemma:qp_norm_bound}
(Simplification of Lemma 12 in \cite{chowdhury2020speeding}) If $(\xb,\yb, \sbb) \in \Ncal_2(\theta)$, then,
$$
 \|\Qb^{-1/2}\pb\|_2
    \leq \sigma \sqrt{\frac{2n\mu}{1-\theta}} + \sqrt{2 n \mu}.
$$
\end{lemma}

\begin{proof}
By Lemma 7 in \cite{chowdhury2020speeding}, if the condition given by eqn. (\ref{eq:sketch_matrix_W}) is fulfilled by sketching matrix $\Wb$, then $\|\Qb^{-1/2} \Ab\Db^2 \Ab^T \Qb^{-1/2} - \Ib_m\|_2 \leq \zeta$. Since $\zeta \in (0,1)$, this imples $\|\Qb^{-1/2} \Ab\Db^2 \Ab^T \Qb^{-1/2}\|_2 \leq 2$ and so $\|\Qb^{-1/2} \Ab\Db\|_2 \leq \sqrt{2}$. Recall that $\pb = - \sigma \mu \Ab \Sb^{-1}\one_n + \Ab \xb.$  We can split the terms of $\Qb^{-1/2}\pb$ by the triangle inequality and then bound these terms separately below. We begin with the first term:
\begin{flalign*}
    \|\sigma \mu  \Qb^{-1/2}\Ab \Sb^{-1}\one_n\|_2
    &= \sigma\mu \|\Qb^{-1/2} \Ab\Db(\Xb\Sb)^{-1/2} \one_n\|_2 \\
    &\leq \sigma\mu\|\Qb^{-1/2} \Ab\Db\|_2 \|(\Xb\Sb)^{-1/2} \one_n\|_2 \\
    &\leq \sqrt{2}\sigma\mu\|(\Xb\Sb)^{-1/2} \one_n\|_2  \\
    &\leq \sqrt{2}\sigma\mu\sqrt{\frac{n}{\min \xb_i\sbb_i}} \\
    &\leq \sqrt{2}\sigma\mu\sqrt{\frac{n}{(1-\theta)\mu}} = \sigma \sqrt{\frac{2n\mu}{1-\theta}}.
\end{flalign*}

Next, we bound the second term:
\begin{flalign*}
    \|\Qb^{-1/2} \Ab\xb\|_2
    &= \|\Qb^{-1/2} \Ab\Db\Db^{-1}\xb\|_2 \\
    &= \|\Qb^{-1/2} \Ab\Db(\Sb^{1/2}\Xb^{-1/2})\Xb \one_n\|_2 \\
    &= \|\Qb^{-1/2} \Ab\Db(\Sb\Xb)^{-1/2} \one_n\|_2 \\
    &\leq \|\Qb^{-1/2} \Ab\Db\|_2 \|(\Sb\Xb)^{-1/2} \one_n\|_2 \\
    &\leq \sqrt{2} \sqrt{\sum_{i=1}^n \xb_i\sbb_i} = \sqrt{2 n \mu}.
\end{flalign*}

By adding the bounds together, we conclude that:
\begin{gather*}
    \|\Qb^{-1/2}\pb\|_2
    \leq \sigma \sqrt{\frac{2n\mu}{1-\theta}} + \sqrt{2 n \mu}.
\end{gather*}

\end{proof}

%% file: 4_pcu.tex
\section{Inexact Predictor-Corrector IPM}
\label{section:pcu}

In this section, we analyze the \textit{inexact} predictor-corrector method (Algorithm~\ref{algo:pcu}) using an inexact linear solver $\solve$.  We follow the proof outline given in Section~\ref{section:overview} to prove the convergence guarantees and time complexity of Algorithm~\ref{algo:pcu}. As is common in predictor-corrector IPMs (see~\cite{wright1997primal}), we will assume that the matrix $\Ab$ has full row rank, ie., $rank(\Ab) = m \leq n$. See Appendix~\ref{section:low_rank} for extensions alleviating this constraint.

\begin{algorithm}[H] 
	\caption{Inexact Predictor-Corrector without correction} \label{algo:pcu}%
	\begin{algorithmic}
		
		\State \textbf{Input:}
		$\Ab\in\R^{m \times n}$, initial feasible point $(\xb^{0},\yb^{0},\sbb^{0}) \in \Ncal_2(0.25)$;  IPM tolerance $\epsilon> 0$, linear solver tolerance $\tolSolve > 0$. 
		
		\vspace{1mm}
		\State \textbf{Initialize:} ~$k\gets 0$;
		
		\vspace{1mm}
		\While{$\mu_k > 2\epsilon$} 

		\vspace{1mm}
		\textit{Predictor Step $(\sigma = 0)$:}
		\State(a) Compute $\dytilde = \solve(\Ab\Db^2\Ab^T, \Ab\xb, \tolSolve)$.
		\vspace{1mm}
		\State (b) Compute $\dxtilde$ and $\dstilde$ using eqn.~(\ref{eq:normal_uncorrected}).
		\vspace{1mm} 
		\State (c) Set $\alpha = \min \left\{
        \nicefrac{1}{2},  ~
        \left( \nicefrac{\mu}{16(\|\dxtilde \circ \dstilde\|_2)} \right)^{1/2}
        \right\}$
		\vspace{1mm}
		\State (d)~Compute $(\xb_k,\yb_k, \sbb_k) = (\xb_k,\yb_k,\sbb_k) + \alpha (\dxtilde_k,\dytilde_k,\dstilde_k)$. 
		\vspace{1mm}
		
		\textit{Corrector Step $(\alpha = 1, \sigma = 1)$:}
		\State(e) Compute $\dytilde = \solve(\Ab\Db^2\Ab^T, -\mu \Ab\Sb^{-1}\one_n + \Ab\xb, \tolSolve)$.
		\vspace{1mm}
		\State (f) Compute $\dxtilde$ and $\dstilde$ using eqn.~(\ref{eq:normal_uncorrected}).
		\vspace{1mm}
		\State (g)~Compute $(\xb_{k+1},\yb_{k+1}, \sbb_{k+1}) = (\xb_k,\yb_k,\sbb_k) + (\dxtilde_k,\dytilde_k,\dstilde_k)$. 
		\vspace{1mm}

		\State (h) $k \gets k + 1$.
		
		\EndWhile
	\end{algorithmic}
\end{algorithm}

We will repeatedly express the inexact step as a function of the exact step. We start by expressing the difference of the steps computed by the \textit{exact} normal equations versus the \textit{inexact} normal equations, i.e. eqns.~(\ref{eq:normal_exact}) versus eqns.~(\ref{eq:normal_uncorrected}):
\begin{align}
    \dyexact - \dytilde &= (\Ab\Db^2\Ab^T)^{-1}\fb, \label{eq:ls_y_diff}\\
    \dsexact - \dstilde &= -\Ab^T(\dyexact - \dytilde)
    = -\Ab^T(\Ab\Db^2\Ab^T)^{-1} \fb, \label{eq:ls_s_diff}\\
    \dxexact - \dxtilde &= -\Db^2(\dsexact - \dstilde)
    = \Db^2\Ab^T(\Ab\Db^2\Ab^T)^{-1} \fb.  \label{eq:ls_x_diff}
\end{align}

We will prove that Algorithm~\ref{algo:pcu} converges to a point $(\xtopt, \ytopt, \stopt)$ satisfying $\mutopt < 2\epsilon$ and $\|\Ab\xtopt - \bb\|_2 < \epsilon$ in $\Ocal(\sqrt{n} \log \nicefrac{\mu_0}{\epsilon})$ outer iterations. First, we start with a technical result to bound $\|\dxtilde \circ \dstilde\|_2$.

\begin{lemma} \label{lemma:pcu_cross_bound}
    Let $(\xb,\yb,\sbb) \in \Ncal_2(\theta)$ and let $(\dxtilde, \dytilde, \dstilde)$ denote step calculated from the inexact normal equations (see eqn.~(\ref{eq:normal_uncorrected})). Then
    \begin{gather*}
        \|\dxtilde \circ \dstilde\|_2
        \leq \frac{\theta^2+n(1-\sigma)^2}{2^{3/2} (1-\theta)} \mu 
        + 2\sqrt{\frac{(\theta^2+n(1-\sigma)^2) \mu}{(1-\theta)}}  \|(\Ab\Db)^\dagger \fb\|_2
        + \|(\Ab\Db)^\dagger \fb\|_2^2.
    \end{gather*}
\end{lemma}

\begin{proof}
We start by expressing the inexact step as the difference from the corresponding exact step, using eqns.~(\ref{eq:ls_s_diff}) and~(\ref{eq:ls_x_diff}).  This will allow us to leverage results for exact predictor-corrector IPMs in our proof:
\begin{flalign*}
    \|\dxtilde \circ \dstilde\|_2 &= \|[\dxexact - (\dxexact - \dxtilde)] \circ [\dsexact - (\dsexact - \dstilde)] \|_2 \\
    &\leq \|\dxexact \circ \dsexact\|_2 + \|\dxexact \circ (\Ab^T(\Ab\Db^2\Ab^T)^{-1} \fb) - \dsexact \circ (\Db^2\Ab^T(\Ab\Db^2\Ab^T)^{-1} \fb)\|_2  \\
    &+ \|(\Ab^T(\Ab\Db^2\Ab^T)^{-1} \fb) \circ (\Db^2\Ab^T(\Ab\Db^2\Ab^T)^{-1} \fb)\|_2.
\end{flalign*}

We will bound each of the three terms in the last inequality separately. Let $\Bcal_1 =  \|\dxexact \circ \dsexact\|_2$; 
$\Bcal_2 =  \|\dxexact \circ (\Ab^T(\Ab\Db^2\Ab^T)^{-1} \fb) - \dsexact \circ (\Db^2\Ab^T(\Ab\Db^2\Ab^T)^{-1} \fb)\|_2$; and 
\newline $\Bcal_3 =  \|(\Ab^T(\Ab\Db^2\Ab^T)^{-1} \fb) \circ (\Db^2\Ab^T(\Ab\Db^2\Ab^T)^{-1} \fb)\|_2$. Using eqn.~(\ref{eq:l2_exact_cross_bound}), we get a bound on $\Bcal_1$: 
\begin{gather*}
    \Bcal_1 \leq \frac{\theta^2+n(1-\sigma)^2}{2^{3/2} (1-\theta)} \mu.
\end{gather*}

Next, we bound $\Bcal_2$.  First, we rearrange $\Bcal_2$ using properties of the Hadamard product (see Lemma \ref{lemma:hadamard_cauchy}) and Moore-Penrose pseudoinverse: 
\begin{flalign}
    \Bcal_2 &= \|\dxexact \circ (\Ab^T(\Ab\Db^2\Ab^T)^{-1} \fb) - \dsexact \circ (\Db^2\Ab^T(\Ab\Db^2\Ab^T)^{-1} \fb)\|_2 \nonumber\\
    &= \| [\Db^{-1}\dxexact] \circ [\Db\Ab^T(\Ab\Db^2\Ab^T)^{-1} \fb] - [\Db\dsexact] \circ [\Db\Ab^T(\Ab\Db^2\Ab^T)^{-1} \fb]\|_2 \nonumber\\
    &= \|[\Db^{-1}\dxexact - \Db\dsexact] \circ [\Db\Ab^T(\Ab\Db^2\Ab^T)^{-1} \fb]\|_2 \nonumber\\
    &\leq \|\Db^{-1}\dxexact - \Db\dsexact\|_2 \|(\Ab\Db)^\dagger \fb\|_2. 
\end{flalign}

We now bound $\|\Db^{-1}\dxexact - \Db\dsexact\|_2$ by using the definitions of $\dyexact$, $\dsexact$ and $\dxexact$ given in eqn.~(\ref{eq:normal_exact}). Thus,
\begin{flalign*}
    \Db^{-1}\dxexact - \Db\dsexact 
    &= \Db^{-1}(-\xb + \sigma \mu \Sb^{-1}\one_n - \Db^2 \dsexact) - \Db\dsexact 
    \\&= \Db^{-1}(-\xb+\sigma\mu \Sb^{-1}\one_n) - 2\Db\dsexact \\
    &= \Db^{-1}(-\xb+\sigma\mu \Sb^{-1}\one_n) - 2(\Ab\Db)^\dagger \Ab (-\xb + \sigma \mu \Sb^{-1}\one_n) \\
    &= (\Db^{-1}-2(\Ab\Db)^\dagger \Ab)(-\xb + \sigma \mu \Sb^{-1}\one_n) \\
    &= (\Db^{-1}-2(\Ab\Db)^\dagger \Ab)\Db\Db^{-1}(-\Xb\one_n + \sigma \mu \Sb^{-1}\one_n) \\
    &= (\Ib-2(\Ab\Db)^\dagger \Ab\Db)[(\Xb\Sb)^{-1/2}(-\Xb\Sb \one_n+\sigma\mu \one_n)].
\end{flalign*}
At this point, we have shown that 
\begin{gather}
    \Bcal_2 \leq \|\Ib-2(\Ab\Db)^\dagger \Ab\Db\|_2 \|(\Xb\Sb)^{-1/2}(-\Xb\Sb \one_n+\sigma\mu \one_n)\|_2 \|(\Ab\Db)^\dagger \fb\|_2.\label{eq:pcu_intermediate_b2}
\end{gather}
We can now use the fact that the two-norm of $\Ib-2(\Ab\Db)^\dagger \Ab\Db = \Ib - (\Ab\Db)^\dagger \Ab\Db - (\Ab\Db)^\dagger \Ab\Db$ is upper bounded by two, since $\Ib - (\Ab\Db)^\dagger \Ab\Db$ and $(\Ab\Db)^\dagger \Ab\Db$ are both projection matrices and have two-norm at most one. Next, we bound the middle term, namely $\|(\Xb\Sb)^{-1/2}(-\Xb\Sb e+\sigma\mu \one_n)\|_2$; the proof will use the fact that $(\xb, \yb, \sbb) \in \Ncal_2(\theta)$, which implies that $\|\xb \circ \sbb - \mu\one_n\|_2 \leq \theta \mu$:

\begin{flalign*}
    \|(\Xb\Sb)^{-1/2}(-\Xb\Sb\one_n+\sigma\mu \one_n)\|_2^2 
    &= \sum_{i=1}^n \frac{(-\xb_i \sbb_i + \sigma \mu)^2}{\xb_i \sbb_i}
    \leq \frac{\|\xb \circ \sbb - \sigma\mu \one_n\|_2^2}{\min_i \xb_i \sbb_i} \\
    &\leq \frac{\|\xb \circ \sbb - \sigma\mu \one_n\|_2^2}{(1-\theta)\mu}
    \leq \frac{\|(\xb\circ \sbb - \mu \one_n) + (1-\sigma)\mu \one_n\|_2^2}{(1-\theta)\mu} \\
    &\leq \frac{[(\xb\circ \sbb - \mu \one_n) + (1-\sigma)\mu \one_n]^T[(\xb\circ \sbb - \mu \one_n) + (1-\sigma)\mu \one_n]}{(1-\theta)\mu} \\
    &\leq \frac{\|\xb\circ \sbb - \mu \one_n\|_2^2 + 2(1-\sigma)\mu \one_n^T (\xb\circ \sbb - \mu \one_n) + (1-\sigma)^2 \mu^2 n}{(1-\theta)\mu} \\
    &\leq \frac{\|\xb\circ \sbb - \mu \one_n\|_2^2 + 2(1-\sigma)\mu (n\mu - n\mu) + (1-\sigma)^2 \mu^2 n}{(1-\theta)\mu} \\
    &\leq \frac{\theta^2\mu^2 + (1-\sigma)^2\mu^2n}{(1-\theta)\mu} 
    \leq \frac{(\theta^2 + (1-\sigma)^2n)\mu}{(1-\theta)}.
\end{flalign*}    

Inserting the bounds for $\|\Ib-2(\Ab\Db)^\dagger \Ab\Db\|_2$ and $\|(\Xb\Sb)^{-1/2}(-\Xb\Sb\one_n+\sigma\mu \one_n)\|_2$ into the previous bound for $\Bcal_2$ gives:
\begin{gather*}
    \Bcal_2 \leq 2 \sqrt{\frac{(\theta^2 + (1-\sigma)^2n)\mu}{(1-\theta)}} \|(\Ab\Db)^\dagger \fb\|_2.
\end{gather*}

Finally, we bound $\Bcal_3$ using properties of the Hadamard product (see Lemma~\ref{lemma:hadamard_cauchy}):   
\begin{flalign}
    \Bcal_3 &= \|[\Ab^T(\Ab\Db^2\Ab^T)^{-1} \fb] \circ [\Db^2\Ab^T(\Ab\Db^2\Ab^T)^{-1} \fb ]  \|_2 \nonumber\\
    &= \|[\Db\Ab^T(\Ab\Db^2\Ab^T)^{-1} \fb] \circ [\Db\Ab^T(\Ab\Db^2\Ab^T)^{-1} \fb ]  \|_2 \nonumber\\
    &= \|\Db\Ab^T(\Ab\Db^2\Ab^T)^{-1}\fb\|_2^2 \nonumber\\
    &= \|(\Ab\Db)^T((\Ab\Db)(\Ab\Db)^T)^{-1} \fb\|_2^2 \nonumber\\
    &= \|(\Ab\Db)^\dagger \fb\|_2^2. 
\end{flalign}

Adding the bounds on $\Bcal_1$, $\Bcal_2$, and $\Bcal_3$ gives the final inequality:
\begin{gather*}
        \|\dxtilde \circ \dstilde\|_2
        \leq \frac{\theta^2+n(1-\sigma)^2}{2^{3/2} (1-\theta)} \mu 
        + 2\sqrt{\frac{(\theta^2+n(1-\sigma)^2) \mu}{(1-\theta)}}  \|(\Ab\Db)^\dagger \fb\|_2
        + \|(\Ab\Db)^\dagger \fb\|_2^2.
    \end{gather*}
    
\end{proof}

The next lemma will bound the value of $\|\xtilde(\alpha) \circ \stilde(\alpha) - \mutilde(\alpha) \one_n\|_2$, which will allow us to later show that the iterates remain in the correct neighborhood for a given step size.

\begin{lemma} \label{lemma:pcu_neighborhood_bound}
If $\alpha \in [0,1]$, then
\begin{equation*}
    \|\xtilde(\alpha) \circ \stilde(\alpha) - \mutilde(\alpha) \one_n\|_2
    \leq (1-\alpha)\|\xb \circ \sbb - \mu \one_n\|_2 + \alpha^2\|\dxtilde \circ \dstilde\|_2 + \alpha^2 \|(\Ab\Db)^\dagger \fb\|_2^2.
\end{equation*}
\end{lemma}

\begin{proof}

First, we derive an expression for the difference in the duality measure between the exact and inexact step, $(\mu(\alpha) - \mutilde(\alpha))$. We begin by expanding the definition of $n(\mu(\alpha) - \mutilde(\alpha))$.
\begin{flalign*}
    n(\mu(\alpha) - \mutilde(\alpha)) &= (\xb+\alpha\dxexact)^T(\sbb+\alpha\dsexact)
    - (\xb+\alpha\Delta \xtilde)^T(\sbb+\alpha\Delta \stilde) \nonumber\\
    &= (\xb^T\sbb + \alpha\dxexact^T\sbb + \xb^T \alpha\dsexact + \alpha^2\dxexact^T \dsexact)
    - (\xb^T\sbb + \alpha\Delta \xtilde^T\sbb + \xb^T \alpha\Delta \stilde + \alpha^2\Delta \xtilde^T \Delta \stilde) \nonumber \\
    &= \alpha(\dxexact - \Delta \xtilde)^T\sbb + \alpha \xb^T(\dsexact - \Delta \stilde) + \alpha^2(\dxexact - \Delta \xtilde)^T(\dsexact - \Delta \stilde) \nonumber\\
\end{flalign*}

We next substitute the differences between the exact and inexact steps given by eqns. (\ref{eq:ls_s_diff}, \ref{eq:ls_x_diff}).  We do this in two parts for clearer exposition.
\begin{flalign*}
     \alpha(\dxexact - \Delta \xtilde)^Ts + \alpha \xb^T(\dsexact - \Delta \stilde)
     &= \alpha[\Db^2\Ab^T(\Ab\Db^2\Ab^T)^{-1} \fb]^T\sbb + \alpha \xb^T[-\Ab^T(\Ab\Db^2\Ab^T)^{-1} \fb] \\
     &= \alpha \fb^T (\Ab\Db^2\Ab^T)^{-1} \Ab \Db^2\sbb
     - \alpha \xb^T\Ab^T(\Ab\Db^2\Ab^T)^{-1} \fb \\
     &= \alpha \fb^T (\Ab\Db^2\Ab^T)^{-1} \Ab \Xb\Sb^{-1}\sbb
     - \alpha \xb^T\Ab^T(\Ab\Db^2\Ab^T)^{-1} \fb \\
     &= \alpha \fb^T (\Ab\Db^2\Ab^T)^{-1} \Ab \xb
     - \alpha \xb^T\Ab^T(\Ab\Db^2\Ab^T)^{-1} \fb \\
     &= 0.
\end{flalign*}

Therefore, substituting this result into the previous set of equations, we get the following.
\begin{flalign}
    n(\mu(\alpha) - \mutilde(\alpha)) 
    &= \alpha^2(\dxexact - \Delta \xtilde)^T(\dsexact - \Delta \stilde) \nonumber\\
    &= \alpha^2[\Db^2\Ab^T(\Ab\Db^2\Ab^T)^{-1} \fb]^T [-\Ab^T(\Ab\Db^2\Ab^T)^{-1} \fb] \nonumber\\
    &= -\alpha^2 \fb^T (\Ab\Db^2\Ab^T)^{-1}(\Ab\Db^2\Ab^T) (\Ab\Db^2\Ab^T)^{-1} \fb \nonumber \\
    &= -\alpha^2 \fb^T (\Ab\Db^2\Ab^T)^{-1} \fb \Rightarrow \nonumber\\
    \mu(\alpha) - \mutilde(\alpha) &= \frac{-\alpha^2}{n}\|(\Ab\Db)^\dagger \fb\|_2^2, \label{eq:pcu_mu_diff}
\end{flalign}
where the final line follows from properties of the pseudoinverse.
Next, we prove two identities that will be useful.  For the first identity, we expand the term $\xb_i(\alpha)\sbb_i(\alpha)$ and insert eqn.~(\ref{eq:exact_mu_alpha_bound}) to cancel terms:
\begin{align}
    \xb_i(\alpha)\sbb_i(\alpha) - \mu(\alpha)
    &= \xb_i \sbb_i + \alpha(\xb_i\dsexact_i + \dxexact_i \sbb_i) + \alpha^2\dxexact_i \dsexact_i - (1-\alpha(1-\sigma))\mu \nonumber\\
    &= (1-\alpha)\xb_i \sbb_i + \alpha \sigma \mu + \alpha^2\dxexact_i \dsexact_i - (1-\alpha+ \alpha\sigma)\mu \nonumber\\
    &=(1-\alpha)(\xb_i \sbb_i - \mu) + \alpha^2 \dxexact_i \dsexact_i.
\end{align}
For the second identity, we expand terms and substitute eqns.~(\ref{eq:ls_s_diff}) and~(\ref{eq:ls_x_diff}) for $\dsexact - \dstilde$ and $\dxexact - \dxtilde$, respectively:
\begin{align}
    &\xb_i(\alpha) \sbb_i(\alpha) - \xtilde(\alpha)_i  \stilde(\alpha)_i\nonumber\\
    &= \xb_i\sbb_i + \alpha(\xb_i\dsexact_i + \dxexact_i s_i) + \alpha^2 \dxexact_i \dsexact_i - [\xb_i\sbb_i + \alpha(\xb_i\dstilde_i + \dxtilde_i \sbb_i) + \alpha^2 \dxtilde_i \dstilde_i] \nonumber\\
    &=\alpha[\xb_i(\dsexact_i-\dstilde_i) + \sbb_i(\dxexact_i-\dxtilde_i)] + \alpha^2(\dxexact_i \dsexact_i - \dxtilde_i \dstilde_i)  \nonumber \\
    &= \alpha[\xb\circ (-\Ab^T(\Ab\Db^2\Ab^T)^{-1} \fb) + \sbb\circ( \Db^2\Ab^T(\Ab\Db^2\Ab^T)^{-1} \fb)]_i + \alpha^2(\dxexact_i \dsexact_i - \dxtilde_i \dstilde_i) \nonumber \\
    &= \alpha[\xb\circ (-\Ab^T(\Ab\Db^2\Ab^T)^{-1} \fb) + \xb\circ(\Ab^T(\Ab\Db^2\Ab^T)^{-1} \fb)]_i + \alpha^2(\dxexact_i \dsexact_i - \dxtilde_i \dstilde_i) \nonumber \\
    &= \alpha^2(\dxexact_i \dsexact_i - \dxtilde_i \dstilde_i).
\end{align}
using the above two identities, we expand and rearrange $\xtilde(\alpha)_i\stilde(\alpha)_i - \mutilde(\alpha)$ to get:
\begin{align*}
    \xtilde(\alpha)_i\stilde(\alpha)_i - \mutilde(\alpha)
    &= \xb(\alpha)_i \sbb(\alpha)_i - (\xb(\alpha)_i \sbb(\alpha)_i - \xtilde(\alpha)_i  \stilde(\alpha)_i) - [\mu(\alpha) - (\mu(\alpha) - \mutilde(\alpha))] \\
    &= [\xb(\alpha)_i \sbb(\alpha)_i - \mu(\alpha) ] - (\xb(\alpha)_i \sbb(\alpha)_i - \xtilde(\alpha)_i  \stilde(\alpha)_i) + (\mu(\alpha) - \mutilde(\alpha)) \\
    &= [(1-\alpha)(\xb_i \sbb_i - \mu) + \alpha^2 \dxexact_i \dsexact_i] - (\xb(\alpha)_i \sbb(\alpha)_i - \xtilde(\alpha)_i  \stilde(\alpha)_i) + (\mu(\alpha) - \mutilde(\alpha)) \\
    &=  (1-\alpha)(\xb_i \sbb_i - \mu) + \alpha^2 \dxexact_i \dsexact_i - \alpha^2(\dxexact_i \dsexact_i - \dxtilde_i \dstilde_i) + (\mu(\alpha) - \mutilde(\alpha)) \\
    &= (1-\alpha)(\xb_i \sbb_i - \mu) + \alpha^2\dxtilde_i \dstilde_i + (\mu(\alpha) - \mutilde(\alpha)).
\end{align*}
Taking vector norms of the above element-wise equality and substituting eqn.~(\ref{eq:pcu_mu_diff}) for $\mu(\alpha) - \mutilde(\alpha)$, we conclude,
\begin{align*}
    \|\xtilde(\alpha) \circ \stilde(\alpha) - \mutilde(\alpha) \one_n\|_2
    &\leq (1-\alpha)\|\xb \circ \sbb - \mu \one_n\|_2 + \alpha^2\|\dxtilde \circ \dstilde\|_2 + \|(\mu(\alpha) - \mutilde(\alpha))\one_n\|_2 \\
    &\leq (1-\alpha)\|\xb \circ \sbb - \mu \one_n\|_2 + \alpha^2\|\dxtilde \circ \dstilde\|_2 + \alpha^2 \|(\Ab\Db)^\dagger \fb\|_2^2.
\end{align*}
\end{proof}

Next, we prove that the predictor stays in the slightly enlarged neighborhood $\Ncal_2(0.5)$ when starting from the smaller neighborhood $\Ncal_2(0.25)$.  We will later show that the corrector step guarantees that the iterate ``returns'' to the ``correct'' $\Ncal_2(0.25)$ neighborhood.  Recall that $\sigma = 0$ in the predictor step to get the following lemma.
\begin{lemma}\label{lemma:pcu_predictor_step_size}
    If $(\xb,\yb,\sbb) \in \Ncal_2(0.25)$,
\begin{equation*}
    \alpha = \min \left\{
    1/2,  \left( \frac{\mu}{16(\|\dxtilde \circ \dstilde\|_2)} \right)^{1/2}
    \right\},
    ~~~~~
    \|(\Ab\Db)^\dagger \fb\|_2 \leq \frac{\sqrt{\mu}}{2},
\end{equation*} 
    then the predictor step $(\xtilde(\alpha), \ytilde(\alpha), \stilde(\alpha)) \in \Ncal_2(0.5)$.
\end{lemma}

\begin{proof}
We begin with the inequality of Lemma~\ref{lemma:pcu_neighborhood_bound}.  Recall that $\|\xb \circ \sbb - \mu \one_n\|_2 \leq \nicefrac{\mu}{4}$ since $(\xb, \yb, \sbb) \in \Ncal(0.25)$:
\begin{align*}
    \|\xtilde(\alpha) \circ \stilde(\alpha) - \mutilde(\alpha) \one_n\|_2
    &\leq (1-\alpha)\|\xb \circ \sbb - \mu \one_n\|_2 + \alpha^2\|\dxtilde \circ \dstilde\|_2 + \alpha^2\|(\Ab\Db)^\dagger \fb\|_2^2 \\
    &\leq (1-\alpha)\|\xb \circ \sbb - \mu \one_n\|_2 + \frac{\mu(\|\dxtilde \circ \dstilde\|_2)}{16(\|\dxtilde \circ \dstilde\|_2} + \frac{1}{2^2} \frac{\mu}{4}\\
    &\leq \frac{(1-\alpha)\mu}{4} + \frac{1}{8(1-\alpha)} (1-\alpha)\mu
    ~~~~(\text{since }(\xb,\yb,\sbb) \in \Ncal_2(0.25)) \\
    &\leq \frac{(1-\alpha)\mu}{4} + \frac{(1-\alpha)\mu}{4}
    ~~~~(\text{since } \alpha \leq \nicefrac{1}{2}) \\
    &\leq \frac{1}{2} (1-\alpha)\mu \leq \frac{1}{2}\mutilde(\alpha).
\end{align*}
The last line follows from eqn.~(\ref{eq:exact_mu_alpha_bound}), which gives $\mu(\alpha) = (1-\alpha)\mu$ for $\sigma=0$.  We also know from eqn.~(\ref{eq:pcu_mu_diff}) that $\mutilde(\alpha) \geq \mu(\alpha)$. Therefore, $\|\xtilde(\alpha) \circ \stilde(\alpha) - \mutilde(\alpha) \one_n\|_2 \leq \nicefrac{1}{2}\,\mutilde(\alpha)$. Now, we must show that the condition $(\xtilde(\alpha), \stilde(\alpha)) > 0$ is fulfilled. First, by eqns. (\ref{eq:exact_mu_alpha_bound}, \ref{eq:pcu_mu_diff}), we have $\mutilde(\alpha) = [1 - \alpha(1-\sigma)]\mu + \frac{\alpha^2}{n}\|(\Ab\Db)^\dagger \fb\|_2^2$, which shows $\mutilde(\alpha') > 0$ for all positive step sizes $\alpha' \leq \alpha$.  From the first part of this proof, we have that $\xtilde_i(\alpha)\stilde_i(\alpha) \geq \nicefrac{1}{2} \mutilde(\alpha)$. We conclude that $(\xtilde(\alpha), \ytilde(\alpha), \stilde(\alpha)) \in \Ncal_2(0.5)$.

\end{proof}

The next lemma shows that the step size given in the previous lemma will result in a multiplicative decrease in the duality gap during the predictor step.

\begin{lemma} \label{lemma:pcu_predictor_step}
If $(\xb,\yb,\sbb) \in \Ncal_2(0.25)$,
\begin{equation*}
        \alpha = \min \left\{
        1/2,  \left( \frac{\mu}{16(\|\dxtilde \circ \dstilde\|_2)} \right)^{1/2}
        \right\},
        ~~~~\text{and}~~~~
        \|(\Ab\Db)^\dagger \fb\|_2 \leq \frac{\sqrt{\mu}}{8},
\end{equation*}
then the predictor step $(\xtilde(\alpha), \ytilde(\alpha), \stilde(\alpha)) \in \Ncal_2(0.5)$ and there exists a constant $C_0 \in (0,1)$ such that,
\begin{equation*}
    \frac{\mutilde(\alpha)}{\mu} \leq 1 - \frac{C_0}{\sqrt{n}}.
\end{equation*}
\end{lemma}
We note that in the above lemma $(\xtilde(\alpha), \ytilde(\alpha), \stilde(\alpha))$ is the output after the predictor step \textit{only}, before the corrector step is applied.
\begin{proof}

To lower bound the decrease in the duality gap at each step, we first find a lower bound for the step size $\alpha$, using the upper bound for $\|\dxtilde \circ \dstilde\|_2$.  By Lemma \ref{lemma:pcu_cross_bound}, we have the following inequality:
\begin{gather*}
    \|\dxtilde \circ \dstilde\|_2
    \leq \frac{\theta^2+n(1-\sigma)^2}{2^{3/2} (1-\theta)} \mu 
    + 2\sqrt{\frac{(\theta^2+n(1-\sigma)^2) \mu}{(1-\theta)}}  \|(\Ab\Db)^\dagger \fb\|_2
    + \|(\Ab\Db)^\dagger \fb\|_2^2.
\end{gather*}

We can simplify this inequality by substituting $\|(\Ab\Db)^\dagger \fb\|_2 \leq \sqrt{\mu}/8$, $\theta = 0.25$, and $\sigma=0$ to get:
\begin{flalign*}
    \|\dxtilde \circ \dstilde\|_2
    &\leq \frac{(0.25)^2+n(1-0)^2}{2^{3/2} (1-0.25)} \mu 
    + 2\sqrt{\frac{(0.25^2+n(1-0)^2) \mu}{(1-0.25)}} \frac{\sqrt{\mu}}{8}
    + \frac{\mu}{64} \\
    &\leq \frac{1/16+n}{2^{3/2} (3/4)} \mu 
    + 2\sqrt{\frac{(1/16+n) \mu}{3/4}} \frac{\sqrt{\mu}}{8}
    + \frac{\mu}{64} \\
    &\leq n\mu\left(\frac{1/16+1}{2^{3/2} (3/4)} 
    + 2\sqrt{\frac{(1/16+1)}{3/4}} \frac{1}{8}
    + \frac{1}{64}\right) \\
    &\leq 0.82\cdot n\mu.
\end{flalign*}
We can now insert the upper bounds for $\|\dxtilde \circ \dstilde\|_2$ and $\|(\Ab\Db)^\dagger \fb\|_2$ to lower bound $\alpha$:
\begin{flalign*}
    \alpha 
    &= \min \left\{1/2,  \left( \frac{\mu}{8(\|\dxtilde \circ \dstilde\|_2 + \|(\Ab\Db)^\dagger \fb\|_2^2)} \right)^{1/2}
    \right\} \\
    &\geq \min \left\{1/2,  \left( \frac{\mu}{8 (0.82) n \mu + \mu/8} \right)^{1/2}
    \right\} \\
    &\geq \min \left\{1/2,  \left( \frac{1}{n(6.56 + 1/8)} \right)^{1/2}
    \right\} \\
    &\geq \sqrt{\nicefrac{0.14}{n}}.
\end{flalign*}
We proceed by using eqns.~(\ref{eq:exact_mu_alpha_bound}) and (\ref{eq:pcu_mu_diff}) to upper bound $\mutilde(\alpha)$. We substitute the upper bound for $\|(\Ab\Db)^\dagger \fb\|_2$ as well as the upper and lower bounds for $\alpha$ as needed to derive a worst case bound. Recall that the step size $\alpha$ is upper bounded by $1/2$ (by definition):
\begin{flalign*}
    \mutilde(\alpha) 
    &\leq \mu(\alpha) - [\mu(\alpha) - \mutilde(\alpha)] 
    \leq (1-\alpha)\mu + \frac{\alpha^2}{n}\|(\Ab\Db)^\dagger \fb\|_2^2 \\
    &\leq \left(1-\alpha\right)\mu + \frac{\alpha^2\mu}{64n} 
    \leq \left(1-\sqrt{\frac{0.14}{n}}\right)\mu + \frac{\mu}{4\cdot64n} \\
    &\leq \left[\left(1-\sqrt{\frac{0.14}{n}}\right) + \frac{1}{4\cdot64\sqrt{n}}\right]\mu
    \leq \left[1 - \frac{\sqrt{0.14} - (1/256)}{\sqrt{n}}\right]\mu.
\end{flalign*}
Then, if we let $C_0 = \sqrt{0.14} - (1/256)$, we get
\begin{gather*}
    \frac{\mutilde(\alpha)}{\mu} \leq 1 - \frac{C_0}{\sqrt{n}},
    ~~~~\text{with }C_0 \in (0,1).
\end{gather*}
\end{proof}

So far, We have shown that the predictor step results in a multiplicative decrease in the duality gap while keeping the next iterate in the neighborhood $\Ncal_2(0.5)$.  We now show that the corrector step returns the iterate to the $\Ncal_2(0.25)$ neighborhood while increasing the duality gap by only a small additive amount. 

\begin{lemma} \label{lemma:pcu_corrector_step}
Let $(\xb, \yb, \sbb) \in \Ncal_2(0.5)$ and
$\|(\Ab\Db)^\dagger \fb\|_2 \leq \nicefrac{\sqrt{\mu}}{2^6}$. Then,
the corrector step $(\xtilde(1), \ytilde(1), \stilde(1)) \in \Ncal_2(0.25)$ and $|\mutilde(1) - \mu| \leq \nicefrac{\|(\Ab\Db)^\dagger \fb\|_2^2}{n}$.
\end{lemma}

\begin{proof}
First, we show that taking a step with step size $\alpha = 1$ and centering parameter $\sigma =1$ from a point in $\Ncal_2(0.5)$ ``returns'' that point to the smaller neighborhood $\Ncal_2(0.25)$. We start with the inequality given by Lemma~\ref{lemma:pcu_neighborhood_bound} with $\alpha = 1$ and then substitute the bound for $\|\dxtilde \circ \dstilde\|_2$ from Lemma~\ref{lemma:pcu_cross_bound} to get:
\begin{align*}
    \|\xtilde(1) \circ \stilde(1) - \mutilde(1)\|_2
    &\leq \|\dxtilde \circ \dstilde\|_2 + \|(\Ab\Db)^\dagger \fb\|_2^2 \\
    &\leq 
    \frac{\theta^2+n(1-\sigma)^2}{2^{3/2} (1-\theta)} \mu 
    + 2\sqrt{\frac{(\theta^2+n(1-\sigma)^2) \mu}{(1-\theta)}}  \|(\Ab\Db)^\dagger \fb\|_2
    + 2\|(\Ab\Db)^\dagger \fb\|_2^2.
\end{align*}
Next, we use $\|(\Ab\Db)^\dagger \fb\|_2 \leq \sqrt{\mu}/2^6$, $\theta = 0.5$, and $\sigma = 1$:
\begin{flalign*}
    \|\xtilde(1) \circ \stilde(1) - \mutilde(1)\|_2
    &\leq \frac{(0.5)^2+n(1-1)^2}{2^{3/2} (1-0.5)} \mu 
    + 2\sqrt{\frac{((0.5)^2+n(1-1)^2) \mu}{(1-0.5)}}  \frac{\sqrt{\mu}}{2^6}
    + \frac{2\mu}{2^{12}} \\
    &\leq \frac{(1/4)}{2^{3/2} (1/2)} \mu 
    + 2\sqrt{\frac{(1/4) \mu}{(1/2)}}  \frac{\sqrt{\mu}}{2^6}
    + \frac{2\mu}{2^{12}} \\
    &\leq \frac{\mu}{2^{9/2}} + \frac{\mu}{2^{11/2}} + \frac{\mu}{2^{11}} \leq \frac{1}{4} \mu.
\end{flalign*}

Again, $\mutilde(\alpha) \geq \mu(\alpha)$ which implies that $\|\xtilde(1) \circ \stilde(1) - \mutilde(1)\|_2 \leq \nicefrac{1}{4}\,\mutilde(1)$, so we can conclude that $(\xtilde(1), \ytilde(1), \stilde(1)) \in \Ncal_2(0.25)$.

Finally, we show that the duality gap increases only slightly.  In the exact case, the duality gap does not increase at all when $\sigma = 1$, i.e. $\mu(1) = \mu$, as can be seen in eqn.~(\ref{eq:exact_mu_alpha_bound}). Therefore, by looking at the difference between the exact and inexact duality gaps, i.e. eqn.~(\ref{eq:pcu_mu_diff}), we get
\begin{gather*}
    |\mutilde(1)-\mu(1)| \leq  \frac{\|(\Ab\Db)^\dagger \fb\|_2^2}{n} 
    \Rightarrow |\mutilde(1)-\mu| \leq \frac{\|(\Ab\Db)^\dagger \fb\|_2^2}{n}.
\end{gather*}
\end{proof}

Finally, we combine the results of the previous lemmas with a standard convergence argument to show the overall correctness and convergence rate of Algorithm~\ref{algo:pcu}, namely the \textit{inexact} Predictor-Corrector IPM, thus proving Theorem~\ref{thm:pcu_final}. 

\begin{proof}(of Theorem \ref{thm:pcu_final})
We introduce the following notation to more easily discuss the steps of Algorithm \ref{algo:pcu}.  Let $(\dxtilde_p, \dytilde_p, \dstilde_p)$ denote the predictor step computed by steps (a-b) and $(\dxtilde_c, \dytilde_c, \dstilde_c)$ denote the corrector step computed by steps (e-f). 

Algorithm~\ref{algo:pcu} first computes the predictor step from $\dytilde_p = \solve(\Ab\Db^2\Ab^T, \Ab\xb, \tolSolve)$ and then computes the corrector step from $\dytilde_c = \solve(\Ab\Db^2\Ab^T, -\mu\Ab\Sb^{-1}\one_n + \Ab\xb, \tolSolve)$.  Let $\fb_p$ and $\fb_c$ denote the error vectors incurred when solving for the predictor and corrector steps, respectively. Guarantee (i) of $\solve$ (see eqn.~(\ref{eq:solve_def})) allows us to bound the term $\|(\Ab\Db)^\dagger \fb\|_2$ as follows: 
\begin{gather}
    \|\dytilde-(\Ab\Db^2\Ab^T)^{-1}(-\sigma \mu \Ab\Sb^{-1}\one_n + \Ab\xb)\|_{\Ab\Db^2\Ab^T} \leq \tolSolve \nonumber\\
    \Rightarrow \|\dytilde-\dyexact\|_{\Ab\Db^2\Ab^T} \leq \tolSolve \nonumber\\
    \Rightarrow \|\dytilde-\dyexact\|_{\Ab\Db^2\Ab^T} \leq \tolSolve \nonumber\\
    \Rightarrow \|(\Ab\Db^2\Ab^T)^{-1} \fb\|_{\Ab\Db^2\Ab^T} \leq \tolSolve \nonumber\\
    \Rightarrow \|(\Ab\Db)^\dagger \fb\|_2 \leq \tolSolve. \label{eq:solve1_bound}
\end{gather}
Thus, we have the following bounds, using the value of $\tolSolve$ in Theorem \ref{thm:pcu_final}:
\begin{gather*}
    \|(\Ab\Db)^\dagger \fb_p\|_2,
    ~
    \|(\Ab\Db)^\dagger \fb_c\|_2 \leq \nicefrac{\sqrt{\epsilon}}{2^6}.
\end{gather*}
Algorithm~\ref{algo:pcu} sets the step size $\alpha$ for the predictor step to the value given by Lemma~\ref{lemma:pcu_predictor_step_size}.  If $(\xb_0, \yb_0, \sbb_0) \in \Ncal_2(0.25)$ and $\mu_0 \geq 2\epsilon$, then Lemma~\ref{lemma:pcu_predictor_step} guarantees that the predictor step will reduce the duality gap by a multiplicative factor of $1 - \nicefrac{C_0}{\sqrt{n}}$, while keeping the iterate in the neighborhood $\Ncal_2(0.5)$. Lemma~\ref{lemma:pcu_corrector_step} then guarantees that the ensuing corrector step will return the iterate to the neighborhood $\Ncal_2(0.25)$, while increasing the duality measure by at most $\nicefrac{\tolSolve}{n}$.  Therefore, a single iteration of Algorithm~\ref{algo:pcu}, starting from a point $(\xb_0, \yb_0, \sbb_0) \in \Ncal_2(0.25)$ such that $\mu_0 \geq 2\epsilon$ guarantees:
\begin{gather*}
    \mu_1 \leq \left(1 - \frac{C_0}{\sqrt{n}}\right)\mu_0 +  \frac{\epsilon C_0}{2 \sqrt{n} \log \mu_0/\epsilon}.
\end{gather*}

This fulfills the conditions of Lemma \ref{lemma:ipm_convergence} with $C_0 \in (0,1)$ and $C_1 \leq \frac{C_0}{2 \sqrt{n} \log \mu_0/\epsilon} \leq \nicefrac{C_0}{\sqrt{n}}$. Therefore, we conclude that Algorithm~\ref{algo:pcu} converges after $\frac{\sqrt{n}}{C_0} \log \frac{\mu_0}{\epsilon}$ outer iterations iterations.

We now prove that the final iterate of Algorithm~\ref{algo:pcu}, which we denote by $(\xtopt, \ytopt, \stopt)$, is $\epsilon$-feasible. In a single iteration, the primal variable changes by $\dxtilde = \dxtilde_p + \dxtilde_c$. Let $\fb = \fb_p + \fb_c$ and recall that by the second guarantee of $\solve$, $\|\Ab\Db^2\Ab^T \dytilde - \pb \|_2 \leq \tolSolve$.  This is equivalent to $\|\fb_p\|_2, \|\fb_c\|_2 \leq \tolSolve$ for all iterations of Algorithm~\ref{algo:pcu}. We proceed by showing that the change in the primal residual at the $k$-th iteration is exactly $\fb_k$: 
\begin{flalign*}
    \|\Ab\xb_k - \Ab\xb_{k-1}\|_2 
    &= \|\Ab(\xb_{k-1} + \dxtilde_{k-1}) - \Ab\xb_{k-1}\|_2
    = \|\Ab\dxtilde_{k-1}\|_2 \\
    &= \|\Ab(\dxexact_{k-1} - \dxtilde_{k-1})\|_2 
    =  \|\Ab(\Db^2\Ab^T(\Ab\Db^2\Ab^T)^{-1} \fb_{k-1})\|_2
    = \|\fb_{k-1}\|_2.
\end{flalign*}

We can use the bound $\|\fb_i\|_2 \leq 2\tolSolve$ to bound the primal residual at the $k$-th iteration.
\begin{flalign*}
    \|\bb - \Ab\xb_k\|_2
    &= \|\bb - \Ab\xb_1 + (\Ab\xb_2 - \Ab\xb_1) + ... + (\Ab\xb_k - \Ab\xb_{k-1})\|_2\\
    &= \|(\Ab\xb_0 - \Ab\xb_1) + (\Ab\xb_2 - \Ab\xb_1) +  ... + (\Ab\xb_k - \Ab\xb_{k-1})\|_2\\
    &\leq \|\Ab\xb_0 - \Ab\xb_1\|_2 + \|\Ab\xb_2 - \Ab\xb_1\| + ... + \|\Ab\xb_k - \Ab\xb_{k-1}\|_2 \\
    &\leq \|\fb_0\|_2 + \|\fb_1\|_2 + ... + \|\fb_{k-1}\|_2
    \leq 2k\tolSolve.
\end{flalign*}

We previously concluded in this proof by Lemma \ref{lemma:ipm_convergence} that the Algorithm \ref{algo:pcu} will converge after $k = \frac{\sqrt{n}}{C_0} \log \frac{\mu_0}{\epsilon}$ iterations.  By the conditions of Theorem \ref{thm:pcu_final}, $\tolSolve < \frac{\epsilon C_0}{2 \sqrt{n} \log \mu_0/\epsilon}$.  Therefore, we conclude that $\|\Ab\xtopt - \bb\|_2 \leq \epsilon$, i.e. the solution is $\epsilon$-primal feasible. 

\end{proof}

%% file: adj_proofs.tex
\section{Error-Adjusted Predictor-Corrector IPM}
\label{section:pcc_proofs}

\begin{lemma} \label{lemma:pcc_cross_bound}
    Let $(\xb,\yb,\sbb) \in \Ncal_2(\theta)$ and let $(\dxtilde, \dytilde, \dstilde)$ be the step calculated from the inexact normal equations with error-adjustment (see eqn.~(\ref{eq:normal_corrected})). Then,
    $$\|\dxtilde \circ \dstilde\|_2 \leq \frac{\theta^2+n(1-\sigma)^2}{2^{3/2} (1-\theta)} \mu + 3 \sqrt{\frac{(\theta^2+n(1-\sigma)^2) \mu}{(1-\theta)}} \|(\Xb\Sb)^{-1/2}\vb\|_2 + 2\|(\Xb\Sb)^{-1/2}\vb\|_2^2.$$
\end{lemma}

\begin{proof}
This proof has a similar structure to Lemma \ref{lemma:pcu_cross_bound}.  However, there are additional steps needed due to the correction vector and we will use several of our previous bounds throughout the proof.  First, we start by expressing $\dxtilde \circ \dstilde$ as a function of the exact step and then we substitute the difference between the exact and error-adjusted steps (eqns.~(\ref{eq:pcc_s_diff}) and~(\ref{eq:pcc_x_diff})):
\begin{flalign*}
    \|\dxtilde \circ \dstilde\|_2 &= \|[\dxexact - (\dxexact - \dxtilde)] \circ [\dsexact - (\dsexact - \dstilde)] \|_2 \\
    &\leq \|\dxexact \circ \dsexact\|_2 + \|\dxexact \circ ( \Ab^T(\Ab\Db^2\Ab^T)^{-1} \Ab\Sb^{-1} \vb) \\
    &~~~~ - \dsexact \circ (\Db^2\Ab^T(\Ab\Db^2\Ab^T)^{-1} \Ab\Sb^{-1} \vb - \Sb^{-1}\vb )\|_2  \\
    &~~~~ + \|(\Db^2\Ab^T(\Ab\Db^2\Ab^T)^{-1} \Ab\Sb^{-1} \vb - \Sb^{-1}\vb ) \circ ( \Ab^T(\Ab\Db^2\Ab^T)^{-1} \Ab\Sb^{-1} \vb)\|_2.
\end{flalign*}

We will bound each of these three terms in the last inequality separately. Let $\Bcal_1 =  \|\dxexact \circ \dsexact\|_2$; $\Bcal_2 =   \|\dxexact \circ ( \Ab^T(\Ab\Db^2\Ab^T)^{-1} \Ab\Sb^{-1} \vb)- \dsexact \circ (\Db^2\Ab^T(\Ab\Db^2\Ab^T)^{-1} \Ab\Sb^{-1} \vb - \Sb^{-1}\vb )\|_2$; and \newline $\Bcal_3 =  \|(\Db^2\Ab^T(\Ab\Db^2\Ab^T)^{-1} \Ab\Sb^{-1} \vb - \Sb^{-1}\vb ) \circ ( \Ab^T(\Ab\Db^2\Ab^T)^{-1} \Ab\Sb^{-1} \vb)\|_2$. Using eqn.~(\ref{eq:l2_exact_cross_bound}), we get a bound on $\Bcal_1$: 
\begin{gather*}
    \Bcal_1 \leq \frac{\theta^2+n(1-\sigma)^2}{2^{3/2} (1-\theta)} \mu.
\end{gather*}
Next, we bound $\Bcal_2$ by splitting it into two parts:
\begin{align*}
    \Bcal_2 
    &=  \|\dxexact \circ ( \Ab^T(\Ab\Db^2\Ab^T)^{-1} \Ab\Sb^{-1} \vb)- \dsexact \circ (\Db^2\Ab^T(\Ab\Db^2\Ab^T)^{-1} \Ab\Sb^{-1} \vb - \Sb^{-1}\vb )\|_2 \\ 
    &= \|\dxexact \circ ( \Ab^T(\Ab\Db^2\Ab^T)^{-1} \Ab\Sb^{-1} \vb)- \dsexact \circ \Db^2\Ab^T(\Ab\Db^2\Ab^T)^{-1} \Ab\Sb^{-1} \vb + \dsexact \circ \Sb^{-1}\vb \|_2 \\
    &\leq \|\dxexact \circ ( \Ab^T(\Ab\Db^2\Ab^T)^{-1} \Ab\Sb^{-1} \vb)- \dsexact \circ \Db^2\Ab^T(\Ab\Db^2\Ab^T)^{-1} \Ab\Sb^{-1} \vb\|_2 + \|\dsexact \circ \Sb^{-1}\vb \|_2.
\end{align*}
Let $\Bcal_2^{(a)} = \|\dxexact \circ ( \Ab^T(\Ab\Db^2\Ab^T)^{-1} \Ab\Sb^{-1} \vb)- \dsexact \circ \Db^2\Ab^T(\Ab\Db^2\Ab^T)^{-1} \Ab\Sb^{-1} \vb\|_2$ and $\Bcal_2^{(b)} = \|\dsexact \circ \Sb^{-1}\vb \|_2$. Notice that $\Bcal_2^{(a)}$ can be bounded in the same way as $\Bcal_2$ was bounded in eqn.~(\ref{eq:pcu_intermediate_b2}) in the predictor-corrector proof without error-adjustment of Section~\ref{section:pcu} by setting $\fb = \Ab\Sb^{-1}\vb$:
\begin{gather}\label{eqn:ppdd1}
    \Bcal_2^{(a)} \leq \|\Ib-2(\Ab\Db)^\dagger \Ab\Db\|_2 \|(\Xb\Sb)^{-1/2}(-\Xb\Sb e+\sigma\mu \one_n)\|_2 \|(\Ab\Db)^\dagger \Ab\Sb^{-1}\vb\|_2.
\end{gather}
Furthermore, the first two terms of the above inequality were already bounded in the predictor-corrector proof without error-adjustment:
\begin{gather*}
    \|\Ib-2(\Ab\Db)^\dagger \Ab\Db\|_2 \leq 2 \quad \text{and} \quad
    \|(\Xb\Sb)^{-1/2}(-\Xb\Sb e+\sigma\mu \one_n)\|_2
    \leq \sqrt{\frac{(\theta^2 + (1-\sigma)^2n)\mu}{(1-\theta)}}.
\end{gather*}
Next, we note that 
$$\|(\Ab\Db)^\dagger \Ab\Sb^{-1}\vb\|_2 = \|(\Ab\Db)^\dagger \Ab\Db (\Xb\Sb)^{-1/2}\vb\|_2 \leq \| (\Ab\Db)^\dagger \Ab\Db \|_2 \|(\Xb\Sb)^{-1/2}\vb\|_2 \leq \|(\Xb\Sb)^{-1/2}\vb\|_2.$$ 
Substituting the above three inequalities into eqn.~(\ref{eqn:ppdd1}) we get:
\begin{gather*}
    \Bcal_2^{(a)} \leq 2 \sqrt{\frac{(\theta^2 + (1-\sigma)^2n)\mu}{(1-\theta)}} \|(\Xb\Sb)^{-1/2}\vb\|_2.
\end{gather*}
We now bound $\Bcal_2^{(b)}$. Using the definition of $\dstilde$ (eqn.~(\ref{eq:pcc_s_diff})), as well as $\|(\Ab\Db)^\dagger \Ab\Db\|_2 \leq 1$ and the bound on $\|(\Xb\Sb)^{-1/2}(-\Xb\Sb\one_n+\sigma\mu \one_n)\|_2$, we get.  Again, we apply Lemma \ref{lemma:hadamard_cauchy} to go from the fourth to the fifth line.
\begin{flalign*}
    \Bcal_2^{(b)} &= \|\Sb^{-1}\vb \circ \dsexact\|_2 \\
    &= \|\Sb^{-1}\vb \circ \Ab^T (\Ab\Db^2\Ab^T)^{-1} (\sigma \mu \Ab \Sb^{-1}\one_n - \Ab\xb)\|_2 \\
    &= \|(\Xb\Sb)^{-1/2}\vb \circ \Db\Ab^T (\Ab\Db^2\Ab^T)^{-1} (\sigma \mu \Ab \Sb^{-1}\one_n - \Ab\xb)\|_2 \\
    &= \|(\Xb\Sb)^{-1/2}\vb \circ (\Ab\Db)^\dagger (\sigma \mu \Ab \Sb^{-1}\one_n - \Ab\xb)\|_2 \\
    &\leq  \|(\Xb\Sb)^{-1/2}\vb\|_2 \|(\Ab\Db)^\dagger (\sigma \mu \Ab \Sb^{-1}\one_n - \Ab\xb)\|_2 \\
    &\leq \|(\Xb\Sb)^{-1/2}\vb\|_2 \|(\Ab\Db)^\dagger \Ab\Db (\sigma \mu  (\Xb\Sb)^{-1/2}\one_n - (\Xb\Sb)^{1/2}\one_n)\|_2 \\
    &\leq \|(\Xb\Sb)^{-1/2}\vb\|_2 \|(\Ab\Db)^\dagger \Ab\Db\|_2 \|(\Xb\Sb)^{-1/2}(-\Xb\Sb\one_n+\sigma\mu \one_n)\|_2 \\
    &\leq  \sqrt{\frac{(\theta^2+n(1-\sigma)^2) \mu}{(1-\theta)}} \|(\Xb\Sb)^{-1/2}\vb\|_2.
\end{flalign*}
We can now use the bounds on $\Bcal_2^{(a)}$ and $\Bcal_2^{(b)}$ to obtain the desired bound on $\Bcal_2$:
\begin{flalign*}
    \Bcal_2 \leq \Bcal_2^{(a)} + \Bcal_2^{(b)} \leq 3\sqrt{\frac{(\theta^2+n(1-\sigma)^2) \mu}{(1-\theta)}} \|(\Xb\Sb)^{-1/2}\vb\|_2.
\end{flalign*}
Finally, we bound $\Bcal_3$.  We distribute the terms in $\Bcal_3$ and split the norm into two components using the triangle inequality:
\begin{align*}
    \Bcal_3 
    &= \|(\Db^2\Ab^T(\Ab\Db^2\Ab^T)^{-1} \Ab\Sb^{-1} \vb + \Sb^{-1}\vb ) \circ (\Ab^T(\Ab\Db^2\Ab^T)^{-1} \Ab\Sb^{-1} \vb)\|_2 \\
    &= \|(\Db^2\Ab^T(\Ab\Db^2\Ab^T)^{-1} \Ab\Sb^{-1} \vb  \circ (\Ab^T(\Ab\Db^2\Ab^T)^{-1} \Ab\Sb^{-1} \vb + \Sb^{-1}\vb \circ (\Ab^T(\Ab\Db^2\Ab^T)^{-1} \Ab\Sb^{-1} \vb\|_2 \\
    &\leq \|(\Db^2\Ab^T(\Ab\Db^2\Ab^T)^{-1} \Ab\Sb^{-1} \vb  \circ \Ab^T(\Ab\Db^2\Ab^T)^{-1} \Ab\Sb^{-1} \vb\|_2 + \|\Sb^{-1}\vb \circ \Ab^T(\Ab\Db^2\Ab^T)^{-1} \Ab\Sb^{-1} \vb\|_2.
\end{align*}
Let $\Bcal_3^{(a)} = \|(\Db^2\Ab^T(\Ab\Db^2\Ab^T)^{-1} \Ab\Sb^{-1} \vb  \circ \Ab^T(\Ab\Db^2\Ab^T)^{-1} \Ab\Sb^{-1} \vb\|_2$ and $\Bcal_3^{(b)} = \|\Sb^{-1}\vb \circ \Ab^T(\Ab\Db^2\Ab^T)^{-1} \Ab\Sb^{-1} \vb\|_2$. We first bound $\Bcal_3^{(a)}$ following similar ideas to the derivation of the bound of $\Bcal_3$  in the predictor-corrector proof without error-adjustment:
\begin{flalign*}
    \Bcal_3^{(a)}
    &= \|(\Db^2\Ab^T(\Ab\Db^2\Ab^T)^{-1} \Ab\Sb^{-1} \vb  \circ \Ab^T(\Ab\Db^2\Ab^T)^{-1} \Ab\Sb^{-1} \vb\|_2 \\
    &= \|(\Db\Ab^T(\Ab\Db^2\Ab^T)^{-1} \Ab\Sb^{-1} \vb  \circ \Db\Ab^T(\Ab\Db^2\Ab^T)^{-1} \Ab\Sb^{-1} \vb\|_2 \\
    &= \|(\Ab\Db)^\dagger \Ab\Sb^{-1}\vb\|_2^2
    \leq \|(\Xb\Sb)^{-1/2}\vb\|_2^2.
\end{flalign*}
We also bound $\Bcal_3^{(b)}$:
\begin{flalign*}
    \Bcal_3^{(b)}
    &= \|\Sb^{-1}\vb \circ \Ab^T(\Ab\Db^2\Ab^T)^{-1} \Ab\Sb^{-1} \vb\|_2
    = \|(\Xb\Sb)^{-1/2}\vb \circ \Db\Ab^T(\Ab\Db^2\Ab^T)^{-1} \Ab\Sb^{-1} \vb\|_2 \\ 
    &\leq \|(\Xb\Sb)^{-1/2}\vb\|_\infty \|(\Xb\Sb)^{-1/2}\vb\|_2
    \leq \|(\Xb\Sb)^{-1/2}\vb\|_2^2.
\end{flalign*}
Combining the above two inequalities, we get the overall bound for $\Bcal_3$:
\begin{align*}
    \|\Bcal_3\|_2 \leq \Bcal_3^{(a)} + \Bcal_3^{(b)} \leq 2\|(\Xb\Sb)^{-1/2}\vb\|_2^2.
\end{align*}
Finally, summing up all bounds for $\Bcal_1$, $\Bcal_2$, and $\Bcal_3$ gives our final inequality:
$$\|\dxtilde \circ \dstilde\|_2 \leq \frac{\theta^2+n(1-\sigma)^2}{2^{3/2} (1-\theta)} \mu + 3 \sqrt{\frac{(\theta^2+n(1-\sigma)^2) \mu}{(1-\theta)}} \|(\Xb\Sb)^{-1/2}\vb\|_2 + 2\|(\Xb\Sb)^{-1/2}\vb\|_2^2.$$
\end{proof}
Recall that a point $(\xb,\yb,\sbb)$ is in the neighborhood $\Ncal_2(\theta)$ if $\|\xb \circ \sbb - \mu\one_n\|_2 \leq \mu\theta$.  We bound the left hand side of this condition after a step of size $\alpha$ is taken. 
\begin{lemma} \label{lemma:pcc_neighborhood_bound}
If $\alpha \in [0,1]$, then
\begin{equation*}
    \|\xtilde(\alpha) \circ \stilde(\alpha) - \mutilde(\alpha) \one_n\|_2
    \leq (1-\alpha)\|\xb \circ \sbb - \mu \one_n\|_2 + \alpha^2\|\dxtilde \circ \dstilde\|_2 + 2\alpha \|\vb\|_2.
\end{equation*} 
\end{lemma}

\begin{proof}
We start by expanding the expression for $\xtilde(\alpha) \circ \stilde(\alpha)$:
\begin{flalign}
	\xtilde(\alpha) \circ \stilde(\alpha) 
	&= (\xb + \alpha \dxtilde) \circ (\sbb + \alpha \dstilde)  \notag \\ 
	& =  \xb \circ \sbb + \alpha(\xb \circ \dstilde + \sbb\circ \dxtilde) + \alpha^2 \dxtilde \circ \dstilde \notag\\ 
	& =  \xb \circ \sbb + \alpha(\xb \circ \dstilde + \sbb \circ (-\xb + \sigma \mu \Sb^{-1}\one_n - \Db^2 \dstilde - \Sb^{-1}\vb)) + \alpha^2 \dxtilde \circ \dstilde \notag\\
	& = \xb \circ \sbb + \alpha(-\xb \circ \sbb + \sigma\mu \one_n - \vb) + \alpha^2\dxtilde \circ \dstilde  \notag \\
	& = (1-\alpha)\xb \circ \sbb + \alpha\sigma\mu \one_n - \alpha\vb + \alpha^2\dxtilde \circ \dstilde. \notag
\end{flalign} 
Left-multiplying the final expression by the vector $\one_n^T$ and dividing by $n$, gives an expression for $\mutilde(\alpha)$.  Notice that $\dxtilde^T\dstilde = -\dxtilde^T \Ab^T \dytilde =  0$ by substituting the definition of the $\dstilde$ without error-adjustment and using the fact that $\Ab \xb = \bb$ at each step:
\begin{flalign}
    \mutilde(\alpha)
    &= \frac{1}{n}\one_n^T[(1-\alpha)\xb \circ \sbb + \alpha\sigma\mu \one_n - \alpha\vb + \alpha^2\dxtilde \circ \dstilde] \nonumber \\
    &= [1 - \alpha(1-\sigma)]\mu - \nicefrac{\alpha}{n} \, \vb^T \one_n. \label{eq:pcc_mu_diff}
\end{flalign}
For simplicity of exposition, we first look at individual elements of the vector $\xtilde(\alpha) \circ \stilde(\alpha) - \mutilde(\alpha)$:
\begin{align*}
    \xtilde_i(\alpha) \stilde_i(\alpha) - \mutilde(\alpha)
    &= (1- \alpha) \xb_i \sbb_i + \alpha \sigma \mu - \alpha \vb_i + \alpha^2 \dxtilde_i \dstilde_i - \mu + \alpha\mu - \alpha \sigma \mu + \frac{\alpha \vb^T \one_n}{n}\\
    &= (1-\alpha)(\xb_i\sbb_i - \mu)  + \alpha^2 \dxtilde_i \dstilde_i - \alpha \left(\vb_i - \frac{\vb^T\one_n}{n} \right).
\end{align*}
We bound the norm of the last summand as follows:
\begin{gather*}
    \frac{1}{n}|\vb^T\one_n| \leq \frac{1}{n} \|\vb\|_2 \|\one_n\|_2 \leq \frac{1}{\sqrt{n}} \|\vb\|_2 \Rightarrow 
    \Big\|\vb - \frac{(\vb^T\one_n)}{n} \one_n\Big\|_2
    \leq \|\vb\|_2 + \frac{1}{\sqrt{n}}\|\vb\|_2  \leq 2 \|\vb\|_2.
\end{gather*}
Therefore, we can conclude that,
\begin{equation*}
    \|\xtilde(\alpha) \circ \stilde(\alpha) - \mutilde(\alpha) \one_n\|_2
    \leq (1-\alpha)\|\xb \circ \sbb - \mu \one_n\|_2 + \alpha^2\|\dxtilde \circ \dstilde\|_2 + 2\alpha \|\vb\|_2.
\end{equation*} 
\end{proof}

Given the previous bound, we can now derive a step size $\alpha$ which guarantees that the iterate remains in $\Ncal_2(0.5)$ after the predictor step.
\begin{lemma} \label{lemma:pcc_predictor_step_size}
   If $(\xb,\yb,\sbb) \in \Ncal_2(0.25)$,
        $\alpha = \min\left\{ \nicefrac{1}{2},
        \left(\nicefrac{\mu}{16 \|\dxtilde \circ \dstilde\|_2} \right)^{1/2} \right\}$, and
    $\|\vb\|_2 \leq \nicefrac{\mu}{32}$,
    then the predictor step  $(\xtilde(\alpha), \ytilde(\alpha), \stilde(\alpha)) \in \Ncal_2(0.5)$.
\end{lemma}

\begin{proof}
Our starting point is the bound of Lemma~\ref{lemma:pcc_neighborhood_bound}. By definition, $\alpha$ is upper-bounded by both $1/2$ and the term depending on $\mu$:
\begin{align*}
    \|\xtilde(\alpha) \circ \stilde(\alpha) - \mutilde(\alpha) \one_n\|_2
    &\leq (1-\alpha)\|\xb \circ \sbb - \mu \one_n\|_2 + \alpha^2\|\dxtilde \circ \dstilde\|_2 + 2 \alpha \|\vb\|_2 \\
    &\leq \frac{(1-\alpha)\mu}{4} + \frac{\mu}{16 \|\dxtilde \circ \dstilde\|_2} \|\dxtilde \circ \dstilde\|_2 + \frac{2}{2} \|\vb\|_2 \\
    &\leq \frac{(1-\alpha)\mu}{4} + \frac{\mu}{16} + \frac{\mu}{32} \\
    &\leq \frac{(1-\alpha)\mu}{4} + \frac{\mu}{8} - \frac{\mu}{32} \\
    &\leq \frac{(1-\alpha)\mu}{4} + \frac{(1-\alpha)\mu}{8(1-\alpha)} - \frac{\mu}{32} \\ 
    &\leq \frac{1}{2}(1-\alpha)\mu - \frac{\mu}{32} \\
    &\leq \frac{1}{2}\mutilde(\alpha).
\end{align*}
The last step follows from eqn. (\ref{eq:pcc_mu_diff}), which states $\mutilde(\alpha) = [1 - \alpha(1-\sigma)]\mu - \nicefrac{\alpha}{n} \, \vb^T \one_n$. By applying the Cauchy-Schwarz inequality to $\vb^T \one_n$ as done previously and $\sigma=0$, we obtain $\mutilde(\alpha) \geq (1 - \alpha)\mu - \frac{\alpha}{\sqrt{n}} \|\vb\|_2$, which allows us to conclude that $\|\xtilde(\alpha) \circ \stilde(\alpha) - \mutilde(\alpha) \one_n\|_2 \leq \frac{1}{2} \mutilde(\alpha)$. Now, we must show that the condition $(\xtilde(\alpha), \stilde(\alpha)) > 0$ is fulfilled. First, by eqns. (\ref{eq:exact_mu_alpha_bound}, \ref{eq:pcc_mu_diff}), we have $\mutilde(\alpha) = [1 - \alpha(1-\sigma)]\mu - \nicefrac{\alpha}{n} \, \vb^T \one_n$, which shows $\mutilde(\alpha') > 0$ for all positive step sizes $\alpha' \leq \alpha$.  From the first part of this proof, we have that $\xtilde_i(\alpha)\stilde_i(\alpha) \geq \nicefrac{1}{2} \mutilde(\alpha)$. We conclude that $(\xtilde(\alpha), \ytilde(\alpha), \stilde(\alpha)) \in \Ncal_2(0.5)$.
\end{proof}
We now show that the predictor step with step size $\alpha$ as given in the above lemma guarantees a multiplicative decrease in the duality gap. Recall that $\sigma = 0$ in the predictor step when solving the normal equations.
\begin{lemma}  \label{lemma:pcc_predictor_step}
If $(\xb,\yb,\sbb) \in \Ncal_2(0.25)$,
     $\alpha = \min\left\{ \nicefrac{1}{2},
        \left(\nicefrac{\mu}{16 \|\dxtilde \circ \dstilde\|_2} \right)^{1/2} \right\}$, and $\|\vb\|_2 \leq \nicefrac{\mu}{32}$,
then the predictor step $(\xtilde(\alpha), \ytilde(\alpha), \stilde(\alpha))$ remains in $\Ncal_2(0.5)$ and there exists a constant $C_0 \in (0,1)$ such that, 
\begin{equation*}
    \frac{\mutilde(\alpha)}{\mu} \leq 1 - \frac{C_0}{\sqrt{n}}.
\end{equation*}
\end{lemma}

\begin{proof}
Lemma~\ref{lemma:pcc_predictor_step_size} already shows that this value of $\alpha$ ensures that the next iterate remains in $\Ncal_2(0.5)$. Therefore, we just need to prove the multiplicative decrease in the duality measure. Towards that end, we will again need to lower bound the step size $\alpha$, starting from the upper bound for $\|\dxtilde \circ \dstilde\|_2$.  By Lemma~\ref{lemma:pcc_cross_bound}, we get the following inequality:
$$\|\dxtilde \circ \dstilde\|_2 \leq \frac{\theta^2+n(1-\sigma)^2}{2^{3/2} (1-\theta)} \mu + 3 \sqrt{\frac{(\theta^2+n(1-\sigma)^2) \mu}{(1-\theta)}} \|(\Xb\Sb)^{-1/2}\vb\|_2 + 2\|(\Xb\Sb)^{-1/2}\vb\|_2^2.$$
We now derive a bound on $\|(\Xb\Sb)^{-1/2}\vb\|_2$ using the bound on $\|\vb\|_2$.  We use the fact that for $(\xb,\yb,\sbb) \in \Ncal_2(\theta)$, $\xb_i\sbb_i \geq (1-\theta) \mu$ to get:\footnote{The inequality $\xb_i\sbb_i \geq (1-\theta) \mu$ follows from the definition of $\Ncal_2(\theta)$.} 
\begin{flalign*}
    \|(\Xb\Sb)^{-1/2}\vb\|_2 
    \leq \|(\Xb\Sb)^{-1/2}\|_2 \|\vb\|_2 
    \leq \frac{1}{\sqrt{\min \xb_i\sbb_i}} \|\vb\|_2 
    \leq \frac{1}{\sqrt{(1/4) \mu}} \frac{\mu}{9} 
    \leq \frac{2\sqrt{\mu}}{9}.
\end{flalign*}
Next, we simplify the inequality from Lemma~\ref{lemma:pcc_cross_bound} by substituting $\|(\Xb\Sb)^{-1/2}\vb\|_2 \leq \frac{2\sqrt{\mu}}{9}$, $\theta = 0.25$, and $\sigma=0$ to get
\begin{flalign*}
    \|\dxtilde \circ \dstilde\|_2 
    &\leq \frac{(0.25)^2+n(1-0)^2}{2^{3/2} (1-0.25)} \mu + 3 \sqrt{\frac{((0.25)^2+n(1-0)^2) \mu}{(1-0.25)}} \frac{2\sqrt{\mu}}{9} + \frac{2\cdot2^2\mu}{9^2} \\
    &\leq \frac{(1/16)+n}{2^{3/2} (3/4)} \mu + 3 \sqrt{\frac{(1/16)+n}{(3/4)}} \frac{2\mu}{9} + \frac{2^3\mu}{9^2} \\
    &\leq n\mu \left( \frac{(1/16)+1}{2^{3/2} (3/4)} + 3 \sqrt{\frac{(1/16)+1}{(3/4)}} \frac{2}{9} + \frac{2^3}{9^2} \right) \\
    &\leq 1.4 \cdot n\mu.
\end{flalign*}
The above upper bound can now be used to lower-bound $\alpha$:
\begin{flalign*}
    \alpha 
    &= \min\left\{ \frac{1}{2}, \left(\frac{\mu}{16 \|\dxtilde \circ \dstilde\|_2} \right)^{1/2} \right\} \\
    &\geq \min\left\{ \frac{1}{2}, \left(\frac{\mu}{16 n\mu(1.4)} \right)^{1/2} \right\} \\
    &\geq \frac{0.2}{\sqrt{n}}.
\end{flalign*}

Eqn.~(\ref{eq:pcc_mu_diff}) states that $\mutilde(\alpha) = [1 - \alpha(1-\sigma)]\mu - \nicefrac{\alpha}{n} \, \vb^T \one_n$. Combining it with our upper and lower bounds for $\alpha$ we can bound the decrease in the duality measure $\mutilde$ as follows:
\begin{flalign}
    \mutilde(\alpha) 
    &= [1 - \alpha(1-\sigma)]\mu - \frac{\alpha}{n} \, \vb^T \one_n \nonumber\\
    &\leq [1 - \alpha(1-\sigma)]\mu + \frac{\alpha}{\sqrt{n}} \|\vb\|_2 \quad\text{(By Cauchy-Schwarz)} \nonumber\\
    &\leq \left(1 - \frac{0.2}{\sqrt{n}}\right)\mu +  \frac{\mu}{2\sqrt{n}\cdot 9} \nonumber \\
    &\leq \left(1 - \frac{0.2 - 1/18}{\sqrt{n}}\right)\mu \nonumber \\
    &\leq \left(1 - \frac{0.14}{\sqrt{n}}\right)\mu. \label{eq:pcc_c0_val}
\end{flalign} 
\end{proof}
We have shown that the predictor step results in a multiplicative decrease in the duality gap, while keeping the next iterate in the neighborhood $\Ncal_2(0.5)$.  We now show that the corrector step returns the iterate to the $\Ncal_2(0.25)$ neighborhood, while increasing the duality gap by a small additive amount. 
\begin{lemma} \label{lemma:pcc_corrector_step}
Let $(\xb, \yb, \sbb) \in \Ncal_2(0.5)$ and $\|\vb\|_2 \leq \nicefrac{\mu}{2^7}$. Then, the corrector step $(\xtilde(1), \ytilde(1), \stilde(1)) \in \Ncal_2(0.25)$ and $|\mutilde(1) - \mu| \leq \frac{1}{\sqrt{n}}\|\vb\|_2$.
\end{lemma}
\begin{proof}
We start by simplifying the inequality of Lemma~\ref{lemma:pcc_cross_bound} for the corrector step. Recall:
$$\|\dxtilde \circ \dstilde\|_2 \leq \frac{\theta^2+n(1-\sigma)^2}{2^{3/2} (1-\theta)} \mu + 3 \sqrt{\frac{(\theta^2+n(1-\sigma)^2) \mu}{(1-\theta)}} \|(\Xb\Sb)^{-1/2}\vb\|_2 + 2\|(\Xb\Sb)^{-1/2}\vb\|_2^2.$$
We bound $\|(\Xb\Sb)^{-1/2}\vb\|_2$ using the bound on $\|\vb\|_2$ from the condition of the lemma:
\begin{flalign*}
    \|(\Xb\Sb)^{-1/2}\vb\|_2 
    \leq \|(\Xb\Sb)^{-1/2}\|_2 \|\vb\|_2
    \leq \frac{1}{\sqrt{\min \xb_i\sbb_i}} \|\vb\|_2
    \leq \frac{1}{\sqrt{(1/2) \mu}} \frac{\mu}{2^7}
    \leq \frac{\sqrt{\mu}}{2^6}.
\end{flalign*}
We then simplify the inequality from Lemma~\ref{lemma:pcc_cross_bound} by substituting $\|(\Xb\Sb)^{-1/2}\vb\|_2 \leq \frac{\sqrt{\mu}}{2^6}$, $\theta = 0.5$, and $\sigma=1$:
\begin{flalign*}
    \|\dxtilde \circ \dstilde\|_2 
    &\leq \frac{(0.5)^2+n(1-1)^2}{2^{3/2} (1-0.5)} \mu + 3 \sqrt{\frac{((0.5)^2+n(1-1)^2) \mu}{(1-0.5)}} \frac{\sqrt{\mu}}{2^6} + 2\frac{\mu}{2^{14}}
    \leq \frac{\mu}{2^{5/2}} + \frac{3\mu}{2^{13/2}} + \frac{\mu}{2^{13}}.
\end{flalign*}
Next, we show that taking a step with step size $\alpha = 1$ and centering parameter $\sigma =1$ from a point in the ``larger'' neighborhood $\Ncal_2(0.5)$ returns the iterate to the ``smaller'' neighborhood $\Ncal_2(0.25)$. We start from the result of Lemma~\ref{lemma:pcc_neighborhood_bound} with $\alpha = 1$:
\begin{align*}
    \|\xtilde(1) \circ \stilde(1) - \mutilde(1)\|_2
    &\leq (1-\alpha)\|\xb \circ \sbb - \mu \one_n\|_2 + \alpha^2\|\dxtilde \circ \dstilde\|_2 + 2 \alpha \|\vb\|_2\\
    &\leq \|\dxtilde \circ \dstilde\|_2 + 2\|\vb\|_2 \\
    &\leq \frac{\mu}{2^{5/2}} + \frac{3\mu}{2^{13/2}} + \frac{\mu}{2^{13}} + \frac{2\mu}{2^7} \\
    &\leq \frac{\mu}{2^{5/2}} + \frac{3\mu}{2^{13/2}} + \frac{\mu}{2^{13}} + \frac{2\mu}{2^7} + \frac{\mu}{2^7} - \frac{\mu}{2^7} \\
    &\leq \frac{\mu}{4} - \frac{\mu}{2^7} \leq  \frac{\mutilde}{4}. 
\end{align*}
The last step follows from $\mutilde(\alpha) \geq \mu - \frac{\alpha}{\sqrt{n}} \|\vb\|_2$, which can be derived from eqn. (\ref{eq:pcc_mu_diff}).

This implies that the corrector step will return the iterate to the neighborhood $\Ncal_2(0.25)$.

Finally, by eqn.~(\ref{eq:pcc_mu_diff}), we know that $\mutilde(\alpha) =  [1-\alpha(1-\sigma)]\mu - \nicefrac{\alpha}{n} \vb^T \one_n$. Substituting $\alpha = 1$ and $\sigma = 1$ allows us to bound $\mutilde(1) - \mu$:
\begin{gather*}
    \mutilde(1) =  \mu - \frac{1}{n} \vb^T \one_n 
    \Rightarrow \mutilde(1) - \mu = \frac{-1}{n} \vb^T \one_n 
    \Rightarrow \mutilde(1) - \mu \leq \frac{1}{n} \sqrt{n}\|\vb\|_2 
    \Rightarrow \mutilde(1) - \mu \leq \frac{1}{\sqrt{n}}\|\vb\|_2.
\end{gather*}
\end{proof}
We are now ready to combine the results of the previous lemmas to show the overall correctness and convergence rate of Algorithm~\ref{algo:pcc}, the error-adjusted inexact predictor-corrector IPM.
\begin{proof}(of Theorem~\ref{thm:pcc_final}) 
By the guarantees of $\solvev$, we know that the error-adjusted normal equations are solved for a given $\sigma$ and $\|\vb\|_2 < \nicefrac{\epsilon}{2^7} < \nicefrac{\mu}{2^7} $ at each iteration. First, Lemma~\ref{lemma:pcc_predictor_step_size} guarantees that the intermediate point computed at step (d) of Algorithm~\ref{algo:pcc} remains in the neighborhood $\Ncal_2(0.5)$. Lemma~\ref{lemma:pcc_predictor_step} guarantees that the predictor step decreases the duality measure of the iterate by at least a multiplicative factor of the form $(1 - \nicefrac{C_0}{\sqrt{n}})$ for some constant $C_0$.

Next, Lemma~\ref{lemma:pcc_corrector_step} ensures that the corrector step of Algorithm~\ref{algo:pcc} returns the iterate to the neighborhood $\Ncal_2(0.25)$, while increasing the duality measure by at most $\nicefrac{1}{\sqrt{n}}\|\vb\|_2 \leq \nicefrac{\tolSolve}{\sqrt{n}}$. Therefore, a single iteration of Algorithm~\ref{algo:pcc} starting from a point $(\xb_0, \yb_0, \sbb_0) \in \Ncal_2(0.25)$ such that $\mu_0 \geq 2\epsilon$ guarantees the following inequality:
\begin{gather*}
    \mu_1 \leq \left(1 - \frac{C_0}{\sqrt{n}}\right)\mu_0 + \frac{\epsilon}{2^7 \sqrt{n}}.
\end{gather*}
This fulfills the conditions of Lemma \ref{lemma:ipm_convergence} with $C_0 \in (0,1)$ and $C_1 \leq \frac{1}{2^7\sqrt{n}} \leq C_0/\sqrt{n}$ (see eqn. (\ref{eq:pcc_c0_val})). Therefore, we conclude that Algorithm \ref{algo:pcc} converges in $\Ocal(\sqrt{n} \log \nicefrac{\mu_0}{\epsilon})$ iterations.

Finally, we prove that the final iterate is primal-feasible, i.e. $\|\Ab\xtopt - \bb\|_2 = 0$.
By eqn.~(\ref{eq:pcc_x_diff}), at each step of Algorithm~\ref{algo:pcc}, $\dxexact - \dxtilde = -\Db^2\Ab^T(\Ab\Db^2\Ab^T)^{-1} \Ab\Sb^{-1} \vb + \Sb^{-1}\vb$. We left multiply this expression by $\Ab$ to get
\begin{gather*}
    \Ab\dxexact - \Ab\dxtilde = -\Ab\Db^2\Ab^T(\Ab\Db^2\Ab^T)^{-1} \Ab\Sb^{-1} \vb + \Ab\Sb^{-1}\vb = \zero_m.
\end{gather*}
This implies that the change in the primal-residual of the error-adjusted algorithm is the same as the change in the primal-residual of the exact algorithm at every iteration. Therefore, since the exact algorithm returns a primal-feasible solution, the error-adjusted algorithm does as well.
\end{proof}

%% file: solver.tex
\section{Implementing \texorpdfstring{$\solve$}{Lg} and \texorpdfstring{$\solvev$}{Lg} using randomized linear algebra}\label{section:solver}

We now discuss how to implement the solvers that are needed in our inexact predictor-corrector IPMs, with and without the correction vector, using standard preconditioned solvers, such as the preconditioned conjugate gradient (PCG) method. We use well-known sketching-based approaches to construct the preconditioner, leveraging results from the randomized linear algebra literature.

We first focus on (full row-rank) constraint matrices $\Ab \in \mathbb{R}^{m \times n}$ that are short-and-fat, ie., $m \ll n$. Clearly such matrices have rank $m \ll n$. In Section~\ref{section:low_rank} below we will discuss how to reduce general LP problems with exact low-rank constraint matrices to this setting. Moreover, in Appendix~\ref{sxn:tall_thin}, we also discuss how to handle the tall-and-thin constraint matrices. Towards that end, consider the LP of eqn.~(\ref{eq:primal_lp_def}) with an input matrix $\Ab \in \mathbb{R}^{m \times n}$ \textit{and} $m \ll n$. First, we prove that the preconditioned conjugate gradient (PCG) method of Algorithm~\ref{algo:PCG} (see also~\cite{chowdhury2020speeding}) can fulfill the requirements of $\solve$ in $\Ocal \left(\log \frac{{\small\sigma_{\max}(\Ab\Db)} n \mu}{\delta} \right)$ iterations and the guarantees of $\solvev$ in $\Ocal \left(\log \frac{n \mu}{\delta} \right)$ iterations. 

Let $\Ab\Db=\Ub\Sigmab\Vb^\ts$ be the thin SVD representation and $\Wb \in \R^{n \times w}$ be an oblivious sparse sketching matrix which satisfies:\footnote{Let $\|\Ab\|_F^2 = \sum_{i,j} A_{ij}^2 =\tr(\Ab^T \Ab)$ denote the (square of the) Frobenius norm of matrix $\Ab$.}
%
\begin{equation}\label{eq:sketch_matrix_W}
    \|\Vb\Wb\Wb^T\Vb^T - \Ib_m\|_2 \leq \frac{\zeta}{2},
\end{equation}
with probability at least $1-\eta$. The work of \citet{cohen2015optimal} shows how to construct such a matrix $\Wb$ fulfilling this guarantee with sketch size $w = \Ocal(\nicefrac{m}{\zeta^2}\cdot\log \nicefrac{m}{\eta})$ and $\Ocal(\nicefrac{1}{\zeta}\cdot\log \nicefrac{m}{\eta})$ non-zero entries per row. One possible construction is to uniformly sample $s = \Ocal(\nicefrac{1}{\zeta}\cdot\log \nicefrac{m}{\eta})$ entries per row of $\Wb$ without replacement and set each of the selected entries to $\pm \nicefrac{1}{s}$ independently and uniformly randomly. Next, we use the above sketching matrix to define $\Qb = \Ab\Db\Wb\Wb^T\Db\Ab^T$; we note that $\Qb$ is not explicitly constructed in Algorithm~\ref{algo:PCG}. Then, with probability at least $1-\eta$, the vector $\ztilde^t$ computed by Algorithm~\ref{algo:PCG} fulfills the following inequality (see Equation 7 in~\cite{chowdhury2020speeding}).
\begin{gather}\label{eq:solver_pcg_decrease}
    \|\Qb^{-1/2} (\Ab\Db^2\Ab^T) \Qb^{-1/2} \ztilde^t -     \Qb^{-1/2}\pb\|_2 \leq \zeta^t \|\Qb^{-1/2} \pb\|_2,
    ~~~\zeta \in(0,1).    
\end{gather}
\begin{algorithm}[H]
	\caption{Preconditioned Conjugate Gradient (Algorithm 1 in \cite{chowdhury2020speeding})}\label{algo:PCG}
	\begin{algorithmic}
		\State \textbf{Input:}
		$\Ab\Db\in\RR{m}{n}$ with $m \ll n$, $\pb\in\R^m$, failure probability $\eta$,  iteration count $t$;
		\vspace{1mm}
		\State 1. Compute $\Ab\Db\Wb$ and its SVD,~where $\Wb \in \R^{n \times w}$ fulfills eqn. (\ref{eq:sketch_matrix_W}) with $r = m$. Let $\Ub_{\Qb} \in \mathbb{R}^{m \times m}$ be the matrix of its left singular vectors and let $\Sigmab_{\Qb}^{\nicefrac{1}{2}} \in \mathbb{R}^{m \times m}$ be the matrix of its singular values; 
		\vspace{1mm}
		\State 2. Compute $\Qb^{-\nicefrac{1}{2}} = \Ub_{\Qb} \Sigmab_{\Qb}^{-\nicefrac{1}{2}}\Ub_{\Qb}^\ts$;
		\vspace{1mm}
		\State 3. Initialize $\tilde{\zb}^{0} \gets \zero_m $ and run standard CG on $\Qb^{-1/2}\Ab\Db^2\Ab^T\Qb^{-1/2}\ztilde = \Qb^{-1/2}\pb$ for $t$ iterations;
		\vspace{1mm}
		\State \textbf{Output:} return $\dytilde = \Qb^{-1/2}\tilde{\zb}^t$;
	\end{algorithmic}
\end{algorithm}
Recall that the function $\solve$ is defined to have the following guarantees:
\begin{gather*}
    \dytilde = \solve(\Ab\Db^2\Ab^T, \pb, \delta)
    \Rightarrow \|\dytilde-(\Ab\Db^2\Ab^T)^{-1}\pb\|_{\Ab\Db^2\Ab^T} \leq \delta
    ~~\text{and}~~
    \|\Ab\Db^2\Ab^T\dytilde - \pb\|_2 \leq \delta.
\end{gather*}
The next lemma shows that Algorithm \ref{algo:PCG} fulfills the conditions of $\solve$.
\begin{lemma}\label{lemma:solve_inner_it}
    If Algorithm \ref{algo:PCG} is used to compute $\dytilde = \solve(\Ab\Db^2\Ab^T, \pb, \delta)$, $(\xb,\yb,\sbb)\in \Ncal_2(\theta)$, and $t = \Ocal \left(\log \frac{{\small\sigma_{\max}(\Ab\Db)} n \mu}{\delta} \right)$, then, with probability at least $1-\eta$, $\dytilde$ satisfies
    \begin{gather*}
    \|\dytilde-(\Ab\Db^2\Ab^T)^{-1}\pb\|_{\Ab\Db^2\Ab^T} \leq \delta
    ~~\text{and}~~
    \|\Ab\Db^2\Ab^T\dytilde - \pb\|_2 \leq \delta.
    \end{gather*}   
\end{lemma}

\begin{proof}
We start by bounding $\|\dytilde-(\Ab\Db^2\Ab^T)^{-1}\pb\|_{\Ab\Db^2\Ab^T}$, where $\dytilde = \Qb^{-1/2}\ztilde^t$, using the guarantee given by eqn.~(\ref{eq:solver_pcg_decrease}) and Lemma 2 of~\cite{chowdhury2020speeding}, which guarantees $(1+\zeta/2)^{-1}\leq \sigma^2_i(\Qb^{-1/2}\Ab\Db) \leq (1-\zeta/2)^{-1}$ for all $i=1\ldots m$, when $\Wb$ fulfills eqn. (\ref{eq:sketch_matrix_W}):
\begin{flalign} 
    \|\Qb^{-1/2} (\Ab\Db^2\Ab^T) \Qb^{-1/2} \zb^{(t)} - \Qb^{-1/2}\pb\|_2 \nonumber
    &=\|\Qb^{-1/2} (\Ab\Db^2\Ab^T) (\Qb^{-1/2} \zb^{(t)} - (\Ab\Db^2\Ab^T)^{-1}\pb)\|_2 \nonumber\\
    &\geq \frac{1}{\sqrt{1+\zeta/2}} \|\Db\Ab^T(\Qb^{-1/2} \zb^{(t)} - (\Ab\Db^2\Ab^T)^{-1}\pb)\|_2 \nonumber\\
    &\geq \frac{1}{\sqrt{1+\zeta/2}}\|(\Qb^{-1/2} \zb^{(t)} -(\Ab\Db^2\Ab^T)^{-1}\pb)\|_{(\Ab\Db^2\Ab^T)}. \nonumber
\end{flalign}
The first step is justified by Lemma \ref{lemma:lower_singular_value_bound}, since the column space of $\Db\Ab^T$ is the row space of $\Qb^{-1/2}\Ab\Db$. We show in Lemma \ref{lemma:qp_norm_bound} that $\|\Qb^{-1/2} \pb\|_2 \leq C\sqrt{n\mu}$ for some constant $C$ depending only on $\sigma$ and $\theta$. Combining this bound with eqn.~(\ref{eq:solver_pcg_decrease}) gives:
\begin{flalign}
    \zeta^t \left(1+\zeta/2 \right)^{1/2} C\sqrt{n\mu}
    &\geq \|(\Qb^{-1/2} \zb^{(t)} -(\Ab\Db^2\Ab^T)^{-1}\pb)\|_{(\Ab\Db^2\Ab^T)}. \label{eq:solve_bound_i}
\end{flalign}
This implies that $\|\dytilde-(\Ab\Db^2\Ab^T)^{-1}\pb\|_{\Ab\Db^2\Ab^T} \leq \delta$ after $t = \Ocal \left(\log \frac{n\mu}{\delta} \right)$ iterations. Next, we bound $\|\Ab\Db^2\Ab^T\dytilde - \pb\|_2$ using the guarantee of eqn.~(\ref{eq:solver_pcg_decrease}):
\begin{align*}
    \|\Qb^{-1/2} (\Ab\Db^2\Ab^T) \Qb^{-1/2} \zb^{(t)} - \Qb^{-1/2}\pb\|_2
    &\geq \sigma_{\min}(\Qb^{-1/2})  \| \Ab\Db^2\Ab^T \Qb^{-1/2} \zb^{(t)} - \pb\|_2 \\
    \Rightarrow \zeta^t \sigma_{\max}(\Qb^{1/2}) \|\Qb^{-1/2} \pb\|_2
    &\geq \| \Ab\Db^2\Ab^T\Qb^{-1/2} \zb^{(t)} - \pb\|_2 \\
    \Rightarrow \zeta^t \sigma_{\max}(\Qb^{1/2})C\sqrt{n\mu}
    &\geq \| \Ab\Db^2\Ab^T\Qb^{-1/2} \zb^{(t)} - \pb\|_2.
\end{align*}
Again, the first step follows from Lemma \ref{lemma:lower_singular_value_bound}. Since $\zeta \in (0, 1)$, we conclude that $\|\Ab\Db^2\Ab^T\dytilde - \pb\|_2 < \delta$ after $t = \Ocal \big(\log \frac{{\small\sigma_{\max}(\Ab\Db)} n \mu}{\delta} \big)$ iterations. Therefore, both guarantees of $\solve$ can be achieved with probability at least $1-\eta$ in $t = \Ocal \big(\log \frac{{\small\sigma_{\max}(\Ab\Db)} n \mu}{\delta} \big)$ iterations. 
\end{proof}

\noindent\textbf{Satisfying eqn.~\eqref{eq:solver_pcg_decrease}.} Exploiting the properties of the preconditioner $\Qb^{-1/2}$,~\cite{chowdhury2020speeding} showed how to satisfy eqn.~\eqref{eq:solver_pcg_decrease}  using popular solvers beyond conjugate gradient. Such solvers include steepest descent and Richardson iteration. We could do the same in our work and prove similar results for, say, the Chebyshev iteration~\cite{barrett1994templates,Gutknecht08,gutknecht2002chebyshev}. Indeed, the preconditioner $\Qb^{-1/2}$ can be combined with Theorem~1.6.2 of~\citep{Gutknecht08} to satisfy eqn.~\eqref{eq:solver_pcg_decrease}. Chebyshev iteration avoids the computation of the inner products which is typically needed for CG or other inexact methods. As a result, Chebyshev iteration offers several advantages in a parallel environment as it does not need to evaluate communication-intensive inner products for computing the recurrence parameters.

\subsection{Computing the error-adjustment vector \texorpdfstring{$\vb$}{Lg} for Algorithm~\ref{algo:pcc}}

In this section we discuss how to efficiently compute the correction vector $\vb$ for our ``corrected'' inexact predictor-corrector IPM.
Recall eqn.~(\ref{eq:normal_corrected}): the correction vector must satisfy $\Ab\Sb^{-1}\vb = (\Ab\Db^2\Ab^T)(\dytilde - \dyexact)$.  One possible construction of such a vector $\vb$ is the following:
\begin{gather}\label{eq:v_def}
    \vb = (\Xb\Sb)^{1/2} \Wb (\Ab\Db\Wb)^\dagger (\Ab\Db^2 \Ab^T \dytilde - \pb).
\end{gather}
Notice that this vector can be constructed efficiently from quantities already computed in Algorithm~\ref{algo:PCG}. Left-multiplying by $\Ab\Sb^{-1}$ immediately proves that this construction for $\vb$ satisfies $\Ab\Sb^{-1}\vb = (\Ab\Db^2\Ab^T)(\dytilde - \dyexact)$, with probability at least $1-\eta$. We now prove the following lemma:
\begin{lemma}\label{lemma:solvev_inner_it}
    Let $\dytilde$ be computed by Algorithm \ref{algo:PCG} and let the correction vector $\vb$ be computed by eqn.~(\ref{eq:v_def}). Then, $\|\vb\|_2 \leq \delta$ after $t = \Ocal \left(\log \frac{n\mu}{\delta} \right)$ iterations. 
\end{lemma}
\begin{proof}
Lemma 5 from~\cite{chowdhury2020speeding} (using our notation) guarantees that:
\begin{gather*}\label{eq:chowdhury_lemma_5}
    \|\vb\|_2 \leq \sqrt{3n\mu} \|\Qb^{-1/2} (\Ab\Db^2\Ab^T) \Qb^{-1/2} \zb^{(t)} - \Qb^{-1/2}\pb\|_2.
\end{gather*}
Using eqn.~(\ref{eq:solve_bound_i}),
\begin{flalign*}
    \|\vb\|_2 
    &\leq \sqrt{3n\mu} \|\Qb^{-1/2} (\Ab\Db^2\Ab^T) \Qb^{-1/2} \zb^{(t)} - \Qb^{-1/2}\pb\|_2 \\
    &\leq \sqrt{3n\mu} (\zeta^t (1+\zeta/2)^{1/2} Cn\sqrt{\mu}).
\end{flalign*}
Again, since $\zeta \in (0,1)$, we can conclude that $\|\vb\|_2 \leq \delta$ after $t = \Ocal \left(\log \frac{n\mu}{\delta} \right)$ iterations of Algorithm~\ref{algo:PCG} with probability at least $1-\eta$.
\end{proof}


\input{low_rank_LP}


\subsection{Running times for Algorithm~\ref{algo:PCG}, \texttt{Solve}, and \texttt{SolveV}}\label{sxn:running_time}

Finally, we discuss the running times of Algorithm~\ref{algo:PCG}, \texttt{Solve}, and \texttt{SolveV}.
\begin{lemma}\label{lemma:pcg_complexity}
    Algorithm \ref{algo:PCG} called with input matrix $\Ab\Db \in \R^{m \times n}$, failure probability $\eta$ and iteration count $t$ has a total time complexity $\Ocal\left(\nnz{\Ab} \cdot \log \nicefrac{m}{\eta} + m^3 \log \nicefrac{m}{\eta} + mt + \nnz{\Ab}\cdot t\right)$.
\end{lemma}
\begin{proof}
    First, recall that $\Wb$ has $\log \nicefrac{m}{\eta}$ non-zero entries per row and $\Db$ is a diagonal matrix.  Therefore, $\Ab\Db\Wb$ can be computed in $\Ocal(\nnz{\Ab}\cdot  \log \nicefrac{m}{\eta})$ time. Then, computing $\Qb^{-1/2}$ via the SVD of $\Ab\Db\Wb$ takes $\Ocal(m^3 \log \nicefrac{m}{\eta})$ time. We conclude that the overall time complexity to compute $\Qb^{-1/2}$ is $\Ocal(\nnz{\Ab} \cdot \log \nicefrac{m}{\eta} + m^3 \log \nicefrac{m}{\eta})$.
    
    After computing the preconditioner, each inner iteration requires multiplying $\ztilde$ with $\Qb^{-1/2}\Ab\Db^2\Ab^T = \Qb^{-1/2}(\Ab\Db)(\Ab\Db)^T$. Multiplying a vector by $(\Ab\Db)(\Ab\Db)^T$ takes $\Ocal(\nnz{\Ab})$ time and multiplying a vector by $\Qb^{-1/2}$ takes $\Ocal(m)$ time. Therefore, the overall time complexity of Algorithm~\ref{algo:PCG} is
    \begin{equation*}
        \Ocal\left(\nnz{\Ab} \cdot \log \nicefrac{m}{\eta} + m^3 \log \nicefrac{m}{\eta} + mt + \nnz{\Ab}\cdot t\right).
    \end{equation*}
\end{proof}
We can then immediately derive the time complexity of $\solve$ by combining Lemma~\ref{lemma:pcg_complexity} and Lemma \ref{lemma:solve_inner_it}. We conclude that $\solve$ can be implemented by Algorithm \ref{algo:PCG} with probability at least $1 - \eta$ in time
\begin{equation}\label{eq:time_solve}
    \Ocal\left(\nnz{\Ab} \cdot \log \nicefrac{m}{\eta} + m^3 \log \nicefrac{m}{\eta} + m\log \frac{{\small\sigma_{\max}(\Ab\Db)} n \mu}{\delta} + \nnz{\Ab} \cdot \log \frac{{\small\sigma_{\max}(\Ab\Db)} n \mu}{\delta} \right).
\end{equation}
We can similarly derive the time complexity of implementing $\solvev$ by combining Lemma~\ref{lemma:pcg_complexity} and Lemma~\ref{lemma:solvev_inner_it}. Observe from eqn.~(\ref{eq:v_def}) that computing $\vb$ does not affect the time complexity, since it is a single matrix-vector product using values already computed by Algorithm~\ref{algo:PCG}, except  pre-multiplying the vector $(\Ab\Db\Wb)^\dagger (\Ab\Db^2 \Ab^T \dytilde - \pb)$ by $\Wb$ that takes time $\Ocal(\nnz{\Ab}.\log \nicefrac{m}{\eta})$ (assuming $\nnz{\Ab}\geq n$), which is dominated by the cost of computing $(\Ab\Db\Wb)^\dagger$.   Therefore, Algorithm~\ref{algo:PCG} combined with eqn.~(\ref{eq:v_def}) can implement $\solvev$ in time,
\begin{equation}\label{eq:time_solvev}
    \Ocal\left(\nnz{\Ab} \cdot \log \nicefrac{m}{\eta} + m^3 \log \nicefrac{m}{\eta} + m\log \frac{n \mu}{\delta} + \nnz{\Ab} \cdot \log \frac{n \mu}{\delta} \right).
\end{equation}
Recall that $\eta$ is the failure probability of the algorithm. 

We note that it is straightforward to obtain an overall time complexity for Algorithms~\ref{algo:pcu} and~\ref{algo:pcc} using the above results by setting $\eta = O\left(\nicefrac{1}{\sqrt{n}\log \left(\nicefrac{\mu_0}{\epsilon}\right)}\right)$ and applying the union bound.

\textbf{Running time for Low-rank constraint matrices.} In Section~\ref{section:low_rank}, the approximate SVD takes $\Ocal(k\cdot\nnz{\Ab})$ time, computing $\Zb^\ts\Ab$ takes another $\Ocal(k\cdot\nnz{\Ab})$ time, and performing the Gaussian elimination to get $k$ linearly independent rows of $\Zb^\ts\Ab$ takes $\Ocal(nk^2)$ time. Therefore, overall it takes $\Ocal(k^2\cdot\nnz{\Ab})$ time to preprocess the data (assuming $\nnz{\Ab}\ge n$). Now, $\solve$ can be implemented in $\Ocal\left(\nnz{\Ab} \cdot \log \nicefrac{k}{\eta} + k^3 \log \nicefrac{k}{\eta} + k\log \frac{{\small\sigma_{\max}(\Ab\Db)} n \mu}{\delta} + \nnz{\Ab} \cdot \log \frac{{\small\sigma_{\max}(\Ab\Db)} n \mu}{\delta} \right)$ time and similarly, $\solvev$ can be implemented in $\Ocal\left(\nnz{\Ab} \cdot \log \nicefrac{k}{\eta} + k^3 \log \nicefrac{k}{\eta} + k\log \frac{n \mu}{\delta} + \nnz{\Ab} \cdot \log \frac{n \mu}{\delta} \right)$ time.

%% file: low_rank_LP.tex
\subsection{Constraint matrices with \texorpdfstring{$m\gg n$}{Lg} and \texorpdfstring{$\rank(\Ab)=n$}{Lg}}\label{sxn:tall_thin}

So far, we only focused on constraint matrices that have full row-rank and are wide \ie,~ $m\ll n$. By considering the dual problem, our methods also address constraint matrices that are tall-and-thin and have full column rank \ie,~ $m\gg n$. Let $\Ab\in\RR{m}{n}$ be the constraint matrix with $m\gg n$ and $\rank(\Ab)=n$ such that the primal LP is given by

\begin{gather}\label{eq:primal}
	\min \cbb^\ts \xb, \text{ subject to } \Ab\xb = \bb, ~\xb \geq \zero.
\end{gather}
The associated dual problem is,
\begin{gather}\label{eq:dual}
	\max \bb^\ts\yb, \text{ subject to } \Ab^\ts\yb + \sbb = \cbb, \sbb \geq \zero,
\end{gather}

Note that the dual variable $\yb$ is a free variable \ie~it can have both non-negative and non-positive entries. However, we can always rewrite $\yb$ as the difference between two non-negative vectors. Therefore, let $\yb=\yb^+ - \yb^-$, where both $\yb^+$, $\yb^-\ge\zero$. Now, if we rewrite eqn.~\eqref{eq:dual} in terms of $\yb^+$ and $\yb^-$ and change $\max \bb^T\yb$ to $\min -\bb^T\yb$, it becomes
\begin{flalign}
\min -\bb^\ts\yb, \text{ subject to } \Ab^\ts\yb^+-\Ab^\ts\yb^- + \sbb = \cbb,\text{~and~~}\yb^+,\yb^-\sbb \geq \zero\label{eq:dualnew}
\end{flalign}

Now, we can express eqn.~\eqref{eq:dualnew} as
\begin{gather}\label{eq:dual3}
	\min \bar{\bb}^\ts\bar{\yb}, \text{ subject to } \bar{\Ab}\bar{\yb} = \cbb, ~\bar{\yb} \geq \zero\,,
\end{gather}
where  $\bar{\Ab}=\begin{bmatrix}
	\Ab^\ts & -\Ab^\ts & ~~\Ib_n
\end{bmatrix}\in\RR{n}{(2m+n)}$, $\bar{\bb}=\begin{pmatrix}
-\bb\\
~\,\,\bb\\
~\,\,\zero_n
\end{pmatrix}\in\mathbb{R}^{2m+n}$, and $\bar{\yb}=\begin{pmatrix}
\,\,\yb^+\\
\,\,\yb^-\\
\sbb
\end{pmatrix}\in\mathbb{R}^{2m+n}$.

\vspace{2mm}
Note that $\bar{\Ab}$ is short-and-fat as $2m+n\gg n$ and it also has full row-rank. Therefore, eqn.~\eqref{eq:dual3} can be solved using our framework.

\subsection{A generalization to low-rank constraint matrices}
\label{section:low_rank}

We will now discuss how to apply randomized preconditioners and iterative solvers to LPs where $\Ab$ can be any $m\times n$ matrix with $\rank(\Ab)=k\ll \min\{m, n\}$, which we assume to be known\footnote{When $k$ is not known in advance, one can efficiently estimate it using trace estimation techniques~\cite{avron2011randomized,ubaru2016fast}.}. In addition, we further emphasize that we also assume the set of primal-dual solutions of the LP is non-empty \ie~there exists at least one feasible point.  

\vspace{-2mm}

First, we briefly discuss the approximate SVD ``proto-algorithm'' of~\cite{halko2011finding} that will be instrumental in translating the low-rank LP into our sketching-based framework. The single-iteration ``proto-algorithm'' of \cite{halko2011finding} returns a matrix $\Zb\in\RR{m}{(\ell+2)}$ with $\Zb^\ts\Zb=\Ib_{(\ell+2)}$ ($\ell\le k$) such that for some constant $\ve_0 \ge 0$, the following inequality holds with high probability:\footnote{Here, $\epsilon_0=9\sqrt{\ell+p}\cdot\sqrt{\min\{m,n\}}$. We set $p=2$ (the minimal allowed value), which suffices for our purposes, since the matrix $\Ab=\Ab_\ell$ has exact low-rank.} 
\begin{flalign}
\|\Ab-\Zb\Zb^\ts\Ab\|_2\le (1+\epsilon_0)\|\Ab-\Ab_\ell\|_2\label{eq:approxsvd}\,,
\end{flalign}
where $\Ab_\ell$ is the \emph{best} $\ell$-rank approximation of $\Ab$. The computation of $\Zb$ is dominated by the cost of multiplying $\Ab$ by a vector and thus, can be computed in $\Ocal(\ell\cdot\nnz{\Ab})$ time. By taking $\ell=k$, we have $\Ab=\Ab_k$ which makes the right hand side of eqn.~\eqref{eq:approxsvd} equal to zero. Therefore, letting $\widetilde{\Ab}=\Zb\Zb^\ts\Ab$ directly yields $\Ab=\widetilde{\Ab}$.

Now, as we already have $\Ab=\widetilde{\Ab}$ from eqn.~\eqref{eq:approxsvd} with $\ell=k$,
$$\mathop{\argmin}\limits_{\Ab\xb=\bb,\,\xb\ge\zero}\cbb^\ts\xb=\mathop{\argmin}\limits_{\widetilde{\Ab}\xb=\bb,\,\xb\ge\zero}\cbb^\ts\xb\,.$$ 

\vspace{-2mm}
Now, the matrix $\Zb$ is orthogonal, so it has full column-rank. Therefore, when multiplying from the left, it keeps the same rank. So $\rank(\Zb^\ts\Ab)=\rank(\Zb\Zb^\ts\Ab)=\rank(\widetilde{\Ab})=k$.

\vspace{-2mm}
Next, let $\xb$ is a feasible point, then
\begin{flalign}
    \bb=\widetilde{\Ab}\xb=\Zb\Zb^\ts\Ab\xb=\Zb\Zb^\ts\bb\label{eq:reduction}\,.
\end{flalign}
Let $\Fcal_1=\{\xb:\widetilde{\Ab}\xb=\bb, \xb\ge\zero\}$ and $\Fcal_2=\{\xb: \Zb^\ts\Ab\xb=\Zb^\ts\bb, \xb\ge\zero\}$ be two sets. If $\ub\in\Fcal_1$, then 
\begin{flalign*}
\Zb^\ts\Ab\ub=(\Zb^\ts\Zb)\Zb^\ts\Ab\ub=\Zb^\ts(\Zb\Zb^\ts\Ab)\ub=\Zb^\ts\widetilde{\Ab}\ub=\Zb^\ts\bb,~~~\ie~\ub\in\Fcal_2\,.
\end{flalign*}
The last equality above holds as $\ub\in\Fcal_1$. Therefore, $\Fcal_1\subseteq\Fcal_2$. Now, we need to prove $\Fcal_2\subseteq\Fcal_1$. For this, let $\ub\in\Fcal_2$. Then, $\widetilde{\Ab}\xb=\Zb(\Zb^\ts\Ab\ub)=\Zb\Zb^\ts\bb=\bb$, where the last equality follows from eqn.~\eqref{eq:reduction}. Therefore, we have $\Fcal_1=\Fcal_2$ \ie~ the feasible region induced by $(\widetilde{\Ab},\bb)$ is identical to the feasible region induced by $(\Zb^\ts\Ab,\Zb^\ts\bb)$. Therefore, the LP $\min_{\widetilde{\Ab}\xb=\bb,\,\xb\ge\zero}\cbb^\ts\xb$ 
can be restated as
\begin{flalign}
	\min\,\cbb^\ts\xb\,,\text{ subject to }\Zb^\ts\Ab\xb=\Zb^\ts\bb\,,\xb\ge \zero\,.\label{eq:primal2}
\end{flalign}
Note that we have already shown $\rank(\Zb^\ts\Ab)=k$ (which is $\ll n$). However, $\Zb^\ts\Ab\in\RR{(k+2)}{m}$ does not have full-rank. Therefore, we can use Gaussian elimination to get the $k$ linearly independent rows of $\Zb^\ts\Ab$ in $\Ocal(nk^2)$ time and solve eqn.~\eqref{eq:primal2} using our framework. See Section~\ref{sxn:running_time} for the running time of our algorithms for low-rank constraint matrices.

%% file: experiments.tex
\section{Experiments}
\label{section:experiments_appendix}

We experimentally validated the key predictions of our results.  First, we measure the number of iterations needed for Algorithm \ref{algo:pcc} to converge in relation to the number of variables $n$, and the precision of the final solution $\epsilon$.

\subsection{Generating the random LP}

To construct a random LP, we first sample $\xb_0 \in \R^n$, $\yb_0 \in \R^m$, and $\Ab \in \R^{m \times n}$, where the entries of $\xb_0$ are sampled uniformly from $[0,10]$ and the entries of $\yb_0$ and $\Ab$ are sampled uniformly from $[-10, 10]$.  We then set $\sbb_0 \in \R^n$ by $[\sbb_0]_i = 20 \cdot [\xb_0]_i^{-1}$.  This guarantees that $\mu_0 = \nicefrac{1}{n}\cdot \sbb_0^T\xb_0 = 20$ and $\|\sbb_0 \circ \xb_0 - \mu_0 \one_n\|_2 = 0$. The generated constraint matrix $\Ab$ and initial primal-dual point $(\xb_0, \yb_0, \sbb_0)$ along with the assumption that the initial point is primal-dual feasible is enough information to exactly describe the linear program.

\subsection{Testing Algorithm \ref{algo:pcc}}

We first test the predictions of Theorem \ref{thm:pcc_final} under a simple instantiation of $\solvev$.  To implement $\solvev$, we sample a random vector $\vb \in \R^n$ randomly from the unit sphere and rescale it so that $\Ab\Sb^{-1}\vb = \delta$, where $\delta$ is the accuracy parameter of $\solvev$.  We then use a standard linear system solver to solve the perturbed system given by Equation \ref{eq:solvev_guarantee}.  Note that this instantiation of $\solvev$ would not be useful in practice, but it is nevertheless useful to test whether the outer iteration complexity of Theorem \ref{thm:pcc_final} holds empirically. Figures \ref{fig:num_it_vs_n} and \ref{fig:num_it_vs_eps} summarize our results on the relationship between the number of outer iterations versus $n$ and $\epsilon$.

We find that, in all displayed experiments, primal infeasibility is around $10^{-10}$ and does not change substantially with $n$ or $\epsilon$.  We conclude that the error-adjustment effectively keeps the iterates primal-feasible, modulo minor numerical errors.

\subsection{Testing an iterative instantiation of $\solvev$ (Algorithm \ref{algo:PCG})}

We repeat the above two experiments while using the iterative linear system solver described in Section \ref{section:solver}.  We note that the iterative solver only requires a few number of iterations $(<20)$ in the parameter regime we test.  We avoid a more in-depth analysis of the PCG iteration complexity, as this was already performed in \cite{chowdhury2020speeding}. Overall, we find that there is no notable difference between using the perturbed $\solvev$ method or the iterative instantiation.  Results of our experiments are summarized in Figures \ref{fig:num_it_vs_n_it} and \ref{fig:num_it_vs_eps_it} below.

\begin{figure}[H]
\begin{minipage}{0.5\textwidth}
    \centering
    \includegraphics[width=\textwidth]{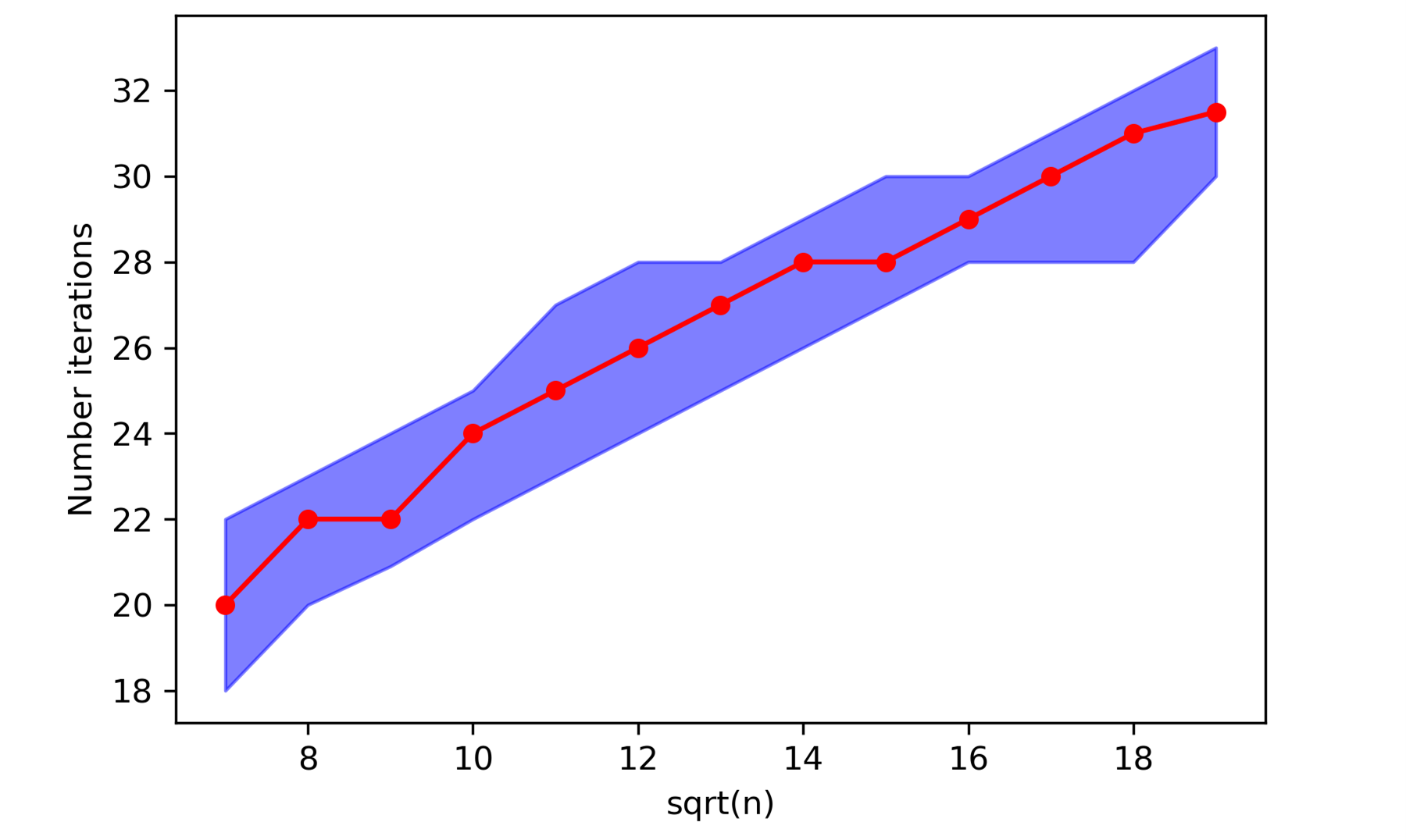}
    \captionof{figure}{This graph demonstrates that the linear relationship between the number of iterations and $\sqrt{n}$ continues to hold when using the iterative instantiation of $\solvev$. The line shows the median number of iterations and the intervals designate the 10\% and 90\% quantiles out of 60 repetitions. Other parameters are $m = 20$, $\epsilon = 0.1$, $\delta = 0.001$, and (sketch size) $w = 60$.}
    \label{fig:num_it_vs_n_it}
\end{minipage}%
\hspace{\columnsep}
\begin{minipage}{0.5\textwidth}
    \centering
    \includegraphics[width=\textwidth]{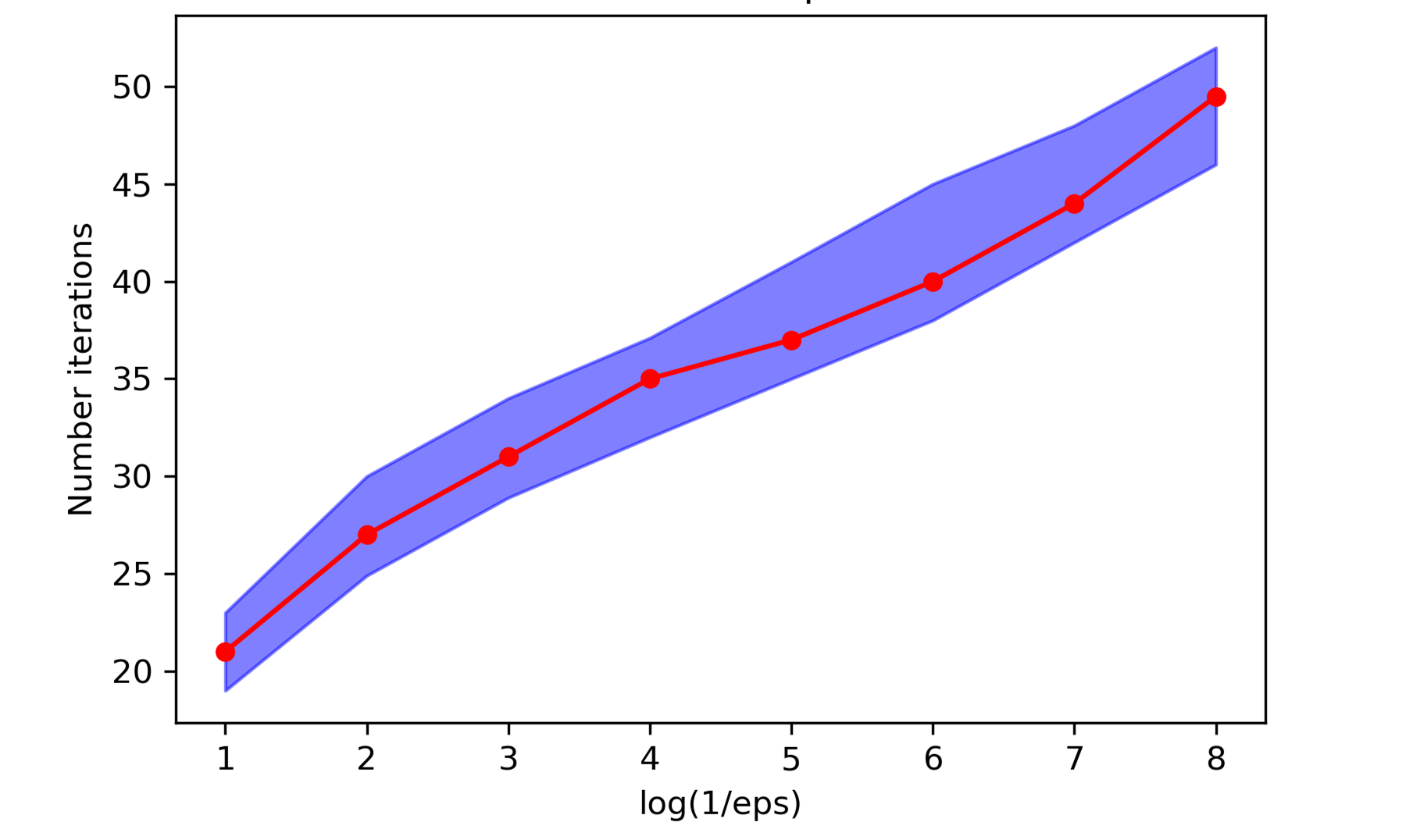}
    \captionof{figure}{This graph demonstrates that the linear relationship between the number of iterations and $\log(1/\epsilon)$ continues to hold when using the iterative instantiation of $\solvev$. The line shows the median number of iterations and the intervals designate the 10\% and 90\% quantiles out of 60 repetitions. Other parameters are $m = 30$, $n = 70$, $\delta(\epsilon) = \epsilon$, and (sketch size) $w=60$.}
    \label{fig:num_it_vs_eps_it}
\end{minipage}
\end{figure}